\documentclass[11pt]{amsart}
\usepackage[utf8]{inputenc}
\usepackage{amsmath, amssymb, amsthm, amsfonts}
\usepackage{subcaption}

\usepackage
[
colorlinks=true,
linkcolor=blue,
anchorcolor=blue,
citecolor=blue,
urlcolor=blue,
plainpages=false,
pdfpagelabels
]{hyperref}

\usepackage{bm}
\usepackage{microtype}

\usepackage[colorlinks=true,linkcolor=blue,anchorcolor=blue,citecolor=blue,urlcolor=blue,plainpages=false,pdfpagelabels]{hyperref}
\hypersetup{linktocpage}
\usepackage{graphicx}
\usepackage{mathrsfs}
\usepackage{enumerate}
\usepackage[capitalize]{cleveref}
\usepackage[noend]{algpseudocode}
\usepackage{comment}
\usepackage[textsize=scriptsize]{todonotes}
\usepackage[capbesideposition=inside,facing=yes,capbesidesep=quad]{floatrow}
\usepackage{multirow}
\usepackage{caption}
\usepackage{blkarray}

\usepackage{xcolor}

\usepackage[all]{xy}
\usepackage{mdwlist}
\usepackage[numbers,sort,compress]{natbib}
\usepackage{scalerel}

\usepackage{ifthen}
\usepackage{xspace}

\usepackage{tikz,tikz-cd}
\usetikzlibrary{matrix,arrows,positioning,calc,decorations.markings}
\tikzset{strike thru arrow/.style={
    decoration={markings, mark=at position 0.5 with {
        \draw [-] 
            ++ (-1pt,-2pt) 
            -- ( 1pt, 2pt);}
    },
    postaction={decorate},
}}
\usepackage{pgfplots}
\usetikzlibrary{intersections, pgfplots.fillbetween}
\pgfdeclarelayer{bg}    
\pgfsetlayers{bg,main}

\newtheorem*{rep@theorem}{\rep@title}

\newcommand{\newreptheorem}[2]{%
\newenvironment{rep#1}[1]{%
 \def\rep@title{#2 \ref{##1}}%
 \begin{rep@theorem}}%
 {\end{rep@theorem}}}
\makeatother
\numberwithin{equation}{section}
\newtheorem{theorem}{Theorem}
\numberwithin{theorem}{section}
\newreptheorem{theorem}{Theorem}

\newtheorem{proposition}[theorem]{Proposition}
\newreptheorem{proposition}{Proposition}
\newtheorem{lemma}[theorem]{Lemma} 

\newreptheorem{corollary}{Corollary}
\newtheorem{conjecture}[theorem]{Conjecture}

\theoremstyle{definition}
\newtheorem{example}[theorem]{Example}
\newtheorem{remark}[theorem]{Remark}
\newtheorem{remarks}[theorem]{Remarks}
\newtheorem{definition}[theorem]{Definition}
\newtheorem{notation}[theorem]{Notation}
\newtheorem*{theorem1i}{\cref{Thm:Main_Thm_Intro} (i)}
\newtheorem*{theorem1iii}{\cref{Thm:Main_Thm_Intro} (iii)}
\newtheorem*{theorem1iv}{\cref{Thm:Main_Thm_Intro} (iv)}

\newcommand\category[1]{\ensuremath{\mathbf{#1}}}

\DeclareMathOperator{\im}{im}

\DeclareMathOperator{\coker}{coker}

\DeclareMathOperator{\id}{Id}

\DeclareMathOperator{\Hom}{Hom}

\DeclareMathOperator{\push}{push}

\DeclareMathOperator{\gr}{gr}

\DeclareMathOperator{\cost}{cost}
\DeclareMathOperator{\labels}{\mathcal L}

\DeclareMathOperator{\Cells}{Cells}
\DeclareMathOperator{\la}{\mathcal L}
\newcommand{\Lan}[0]{E}

\newcommand{\Vect}[0]{\category{Vect}}

\newcommand{\Top}[0]{\category{Top}}

\newcommand{\R}[0]{\mathbb R}

\newcommand{\B}[1]{\mathcal B\ifthenelse{\equal{#1}{}}{}{_{#1}}}
\newcommand{\Bopen}[1]{U\ifthenelse{\equal{#1}{}}{}{(#1)}}

\newcommand{\barc}[1]{\mathcal B\ifthenelse{\equal{#1}{}}{}{(#1)}}

\newcommand{\C}[0]{\mathcal C}

\newcommand{\G}[0]{{\mathcal X}}
\newcommand{\I}[0]{I}
\newcommand{\mcI}[0]{\mathcal I}

\newcommand{\W}[0]{\mathcal W}

\newcommand{\Z}[0]{\mathbb Z}

\newcommand{\pfd}{p.f.d.\@\xspace}

\newcommand{\match}{\mathcal{M}}

\usepackage{trimclip}

\makeatletter
\DeclareRobustCommand{\subto}{%
  \mathrel{\mathpalette\short@to\relax}%
}

\newcommand{\short@to}[2]{%
  \mkern2mu
  \clipbox{{.35\width} 0 0 0}{$\m@th#1\vphantom{+}{\rightarrow}$}%
  }
\makeatother

\title[$\ell^p$-Distances on Multiparameter Persistence Modules]{$\ell^p$-Distances on Multiparameter\\ Persistence Modules}
\author{H\aa vard Bakke Bjerkevik}
\address{Institute of Geometry, TU Graz}
\email{bjerkevik@tugraz.at}
\author{Michael Lesnick}
\address{Department of Mathematics and  Statistics, University at Albany, SUNY}
\email{mlesnick@albany.edu}

\date{}

\begin{document}
\begin{abstract}
Motivated both by theoretical and practical considerations in topological data analysis, we generalize the $p$-Wasserstein distance on barcodes to multiparameter persistence modules. For each $p\in [1,\infty]$, we in fact introduce two such generalizations $d_{\mcI}^p$ and $d_\match^p$, such that $d_{\mcI}^\infty$ equals the interleaving distance and $d_{\match}^\infty$ equals the matching distance.  We show that on 1- or 2-parameter persistence modules over prime fields, $d_{\mcI}^p$ is the universal (i.e., largest) metric satisfying a natural stability property; this 
extends a stability theorem of Skraba and Turner for the $p$-Wasserstein distance on barcodes in the 1-parameter case, and is also a close analogue of a universality property for the interleaving distance given by the second author.  
We also show that $d_\match^p\leq d_{\mcI}^p$ for all $p\in [1,\infty]$, extending an observation of Landi in the $p=\infty$ case.  We observe    
 that on 2-parameter persistence modules, $d_\match^p$ can be efficiently approximated.  
In a forthcoming companion paper, we apply some of these results to study the stability of ($2$-parameter) multicover persistent homology.
\end{abstract}
\maketitle
{\small
\tableofcontents
}
\vspace{-3ex}
\section{Introduction}
Topological data analysis (TDA) extracts multiscale information about the geometry of a data set by constructing a diagram of topological spaces from the data.  The standard example is persistent homology \cite{zomorodian2005computing,edelsbrunner2002topological}, which constructs a \emph{filtration}, i.e., a diagram indexed by a totally ordered set $T$,  and then applies homology with field coefficients to obtain descriptors of the data called \emph{barcodes}.  A barcode is simply a collection of intervals in $T$.  In the last two decades, persistent homology has been the subject of intense interest, leading to a rich theory, highly efficient algorithms and software, and hundreds of applications \cite{otter2017roadmap,carlsson2009topology,rabadan2019topological,giunti2021TDA}.

However, in many settings, e.g., in the study of noisy or time-varying data, a single filtration cannot fully capture the structure of interest in the data.  
We are then led to consider \emph{multiparameter persistent homology}, where we instead construct a diagram of spaces indexed by a product of $d\geq 2$ totally ordered sets; the case $d=2$ is of particular interest in applications  \cite{carlsson2009theory}.  Applying homology with field coefficients to this diagram yields a diagram of vector spaces called a \emph{(multiparameter) persistence module}.  Whereas the isomorphism type of a 1-parameter persistence module is completely described by a barcode, this is not the case for 2 or more parameters, and no fully satisfactory generalization of a barcode is available in the multiparameter setting \cite{carlsson2009theory,botnan2021signed}.  This creates substantial challenges for the development of multiparameter persistent homology as a practical data analysis methodology.  A great deal of recent research has been aimed at addressing these challenges, e.g., \cite{botnan2021signed,miller2020homological,lesnick2015theory,blumberg2020stability,scolamiero2017multidimensional,rolle2020stable,kerber2019exact,corbet2019kernel,dey2019generalized} and the progress has been very encouraging.  While practical applications of multiparameter persistence are still in their early stages, a few promising applications have been developed, e.g., in \cite{carriere2020multiparameter,vipond2021multiparameter}, and with the recent advent of efficient algorithms and software for multiparameter persistence \cite{kerber2021fast,lesnick2019computing,kerber2020efficient,lesnick2015interactive} there seems to be great potential for further applications.

Much of both the theory and applications of 1-parameter persistence use barcodes primarily via distances defined on the space of barcodes.  Analogously, to develop theory and applications of multiparameter persistent homology, one needs suitable distances on multiparameter persistence modules.  Distances on multiparameter persistence modules have thus been a major focus of recent work, and many have been proposed, e.g., in \cite{lesnick2015theory,cerri2013betti,scolamiero2017multidimensional,bubenik2018wasserstein,cerri2019geometrical,thomas2019invariants,mccleary2020edit,giunti2021amplitudes}.  Of these, the distances that have received the most attention and (arguably) have the most fully developed theory are \emph{$\ell^\infty$-distances}, i.e., they can be defined in terms of $\ell^\infty$-metrics on Euclidean spaces, and they behave accordingly; we discuss these distances in detail below.  But as we will explain, there is a clear need in both theory and applications for $\ell^p$-distances on multiparameter persistence modules.  In spite of some prior work in this direction (see \cref{Sec:Related_Work}), our understanding of such distances has lagged well behind our understanding of $\ell^\infty$-distances.

In this paper, we develop a theory of $\ell^p$-distances on multiparameter persistence modules which extends and closely parallels the existing theory for $\ell^\infty$-distances, and at the same time extends fundamental ideas about $\ell^p$-distances on barcodes in the 1-parameter setting.  We expect the ideas introduced here to be useful in both theory and applications of multiparameter persistent homology.

\subsection{Generalizations of the Bottleneck Distance}\label{Sec:Generalizations}
In the theory of persistent homology, the most widely used distance on barcodes is the \emph{bottleneck distance} $d_{\W}^\infty$.  It is used to state standard stability results  for persistent homology \cite{cohen2007stability,chazal2009proximity,bauer2015induced,chazal2009gromov,chazal2014persistence}, and it plays an important role in the statistical foundations of the subject \cite{fasy2014confidence}, as well as in the computational theory \cite{sheehy2013linear,cavanna2015geometric}.

In the TDA literature, $d_{\W}^\infty$ is generalized in two key directions, as summarized by the middle row and middle column of \cref{Table:Generalizations_of_Bottleneck}.  
\begin{table}[h]
\captionsetup{width=0.8\textwidth}
\def\arraystretch{1.8}%
  \begin{tabular}{ l | c | c c }
     &  $p=\infty$ &  $p\in [1,\infty]$ &\\
      \cline{1-3}
     $1$-parameter & $d_{\W}^\infty$& $d_{\W}^p$ &  \parbox[t]{5mm}{\multirow{2}{*}{\rotatebox[origin=c]{270}{$\xrightarrow{\mathrm{\large generality}}$}}}    \\
         \cline{1-3}
       $n$-parameter, $n\geq 1$ & $d_\match\leq d_{\mcI}$ & $\textcolor{red}{d^p_\match\leq d^p_\mcI}$ & \\
          \multicolumn{1}{c}{} &\multicolumn{2}{c}{$\xrightarrow{\mathrm{\large generality}}$} &
    \label{Table:Distances}
  \end{tabular}
\caption{The distances on persistence modules considered in this paper. The distances we introduce appear in the bottom right corner, in red.  As one moves rightward or downward through the cells of the table, the distances become more general.}
\label{Table:Generalizations_of_Bottleneck}
\end{table}
\\ 
First, $d_{\W}^\infty$ generalizes to the \emph{$p$-Wasserstein distance} $d^p_\W$ on barcodes, for any $p\in [1,\infty]$  \cite{cohen2010lipschitz,robinson2013hypothesis,carlsson2004persistence}; this amounts to replacing an $\ell^\infty$-norm appearing in the definition of $d_{\W}^\infty$ with an $\ell^p$-norm.  In fact, several variants of the definition of $d^p_\W$ appear in the literature, which differ in a choice of $\ell^q$-norm on $\R^2$.  Following \cite{robinson2013hypothesis,skraba2020wasserstein}, we take $p=q$ in this paper.  See \cref{Sec:Wasserstein_Distance} for the definition of $d^p_\W$ and its variants. 

Many practical applications of persistent homology involve the computation of $p$-Wasserstein distances on barcodes, usually with $p\in\{ 1, 2,\infty\}$.  Efficient algorithms and code are available for this \cite{kerber2017geometry,chen2021approximation}.   
In spite of the importance of $d_{\W}^\infty$ to the theory, applications using a small value of $p$ appear to be more common; papers describing practical applications of $d_{\W}^1$ or $d_{\W}^2$ include \cite{robinson2013hypothesis,gamble2010exploring,dequeant2008comparison,cang2018representability, gidea2017topological,berwald2014critical,belchi2018lung,biwer2017windowed,robinson2018geometry,bramer2020atom,cang2020evolutionary,yalniz2020inferring,Mukherjee2021.04.26.441473,hamilton2021persistent,khalil2021topological}.  Loosely speaking, as $p$ decreases the $p$-Wasserstein distances become relatively less sensitive to outlying intervals in the barcodes, and more sensitive to the number of small intervals, which may explain in part why small choices of $p$ are often preferred.

On the theoretical side, the 1-Wasserstein distance on barcodes is often used to formulate stability results for ``vectorizations'' of barcodes, i.e., for maps from barcode space to a linear space which are used in machine learning applications of TDA; see, e.g., \cite[Section 7]{skraba2020wasserstein} and the references given there.  

As the second key direction of generalization, $d_{\W}^\infty$ extends to a pair of distances on multiparameter persistence modules, the \emph{interleaving distance} $d_{\mcI}$ \cite{chazal2009proximity,chazal2012structure} and the \emph{matching distance} $d_\match$ \cite{cerri2013betti}.   The interleaving distance is the most prominent distance in the literature on multiparameter persistence modules.  It has proven to be a useful theoretical tool in TDA, e.g., for formulating stability and consistency results \cite{lesnick2012multidimensional,blumberg2020stability,rolle2020stable}.  The use of $d_{\mcI}$ in the TDA theory is justified in part by a universal property which says that on multiparameter persistence modules over prime fields, $d_{\mcI}$ is the largest distance satisfying a natural stability property \cite{lesnick2015theory}; see \cref{Thm:Stability_and_Universality_of_The_Interleaving_Distance} below.
 
The result that $d_{\mcI}=d_{\W}^\infty$ on 1-parameter persistence modules is known as the \emph{isometry theorem} \cite{bauer2015induced,bjerkevik2016stability,chazal2012structure,lesnick2015theory,bubenik2014categorification}; see \cref{Thm:Isometry} for a formal statement.  The inequality $d_{\mcI}\geq d_{\W}^\infty$, originally due to Chazal et al. \cite{chazal2009proximity}, is called the \emph{algebraic stability theorem}.  It plays a particularly important role in persistence theory.  
 
It has been shown that computing $d_{\mcI}$ on $n$-parameter persistence modules is NP-hard for $n\geq 2$ \cite{bjerkevik2019computing}.  This motivates the search for a computable surrogate for $d_{\mcI}$, and the matching distance has emerged as a natural choice.  It has been shown by Landi \cite{landi2018rank} (and is implicit in earlier work by Cerri et al. \cite{cerri2013betti}) that $d_\match\leq d_{\mcI}$. 
On bipersistence modules, (i.e., 2-parameter persistence modules,) $d_\match$ is computable in polynomial time \cite{kerber2019exact}.  It can also be efficiently approximated \cite{biasotti2011new,kerber2020efficient}, and a fast implementation of the approximation algorithm of \cite{kerber2020efficient} was recently made available as part of the Hera code \cite{kerber2017geometry}.  It is expected that these results about computing $d_\match$ on bipersistence modules extend to higher-parameter persistence modules, though this has not been proven. 
While there have been only a few applications of the multidimensional matching distance to date \cite{keller2018persistent,biasotti2011new}, it seems likely that recent advances in 2-parameter persistence computation and software \cite{lesnick2019computing,kerber2020efficient,kerber2021fast} will lead to further applications.

\subsection{Stability of the Interleaving and Wasserstein Distances}\label{Sec:Stability_Intro}
The $p$-Wasserstein distances and interleaving distance satisfy similar stability properties, which will be important points of reference for our work.  
To state these properties, we will need a few definitions.  We consider $\R^n$ as a poset with the usual product partial order, i.e., $(a_1,\ldots,a_n)\leq (b_1,\ldots, b_n)$ if and only if $a_i \leq b_i$ for each $i$.   Given a CW-complex $X$, let $\Cells(X)$ denote the set of cells of $X$.  Following \cite{skraba2020wasserstein}, we say a function $f\colon\Cells(X)\to \R^n$ is \emph{monotone} if $f(\sigma)\leq f(\tau)$ whenever the attaching map of $\tau$ has non-trivial intersection with $\sigma$.  

\begin{definition}[Sublevel filtrations]\mbox{}
\begin{itemize}
\item[(i)] For any topological space $T$ and function $f\colon T\to \R^n$, the sublevel filtration of $f$, denoted $\mathcal S(f)$, is the functor $\R^n\to \Top$ given by 
\[\mathcal S(f)_a=\{x\in T\mid f(x)\leq a\}\]
for each $a\in \R^n$, with the internal maps of $\mathcal S(f)$ taken to be inclusions.  
\item[(ii)] Similarly, given a CW-complex $X$ and monotone $f\colon\Cells(X)\to \R^n$, the sublevel filtration of $f$, also denoted $\mathcal S(f)$, is the functor $\R^n\to \Top$ given by taking $\mathcal{S}(f)_a$ to be the following subcomplex of $X$ 
\[\mathcal{S}(f)_a=\{\sigma\in \Cells(X)\mid f(\sigma)\leq a\},\]
with the internal maps again taken to be inclusions.  
\end{itemize}
\end{definition}
For $p\in [1,\infty]$ and $v\in \R^n$, let $\|v\|_p$ denote the $p$-norm of $v$. 

\begin{notation}\label{Not:function_p_norm}
Given a set $S$ and function $f\colon S\to \R^n$, let 
\[\|f\|_\infty = \sup_{x\in S}\| f(x)\|_\infty,\]
and if $f$ has finite support, then for all $p\in [1,\infty)$ let
\[\|f\|_p= \left(\sum_{x\in S}\|f(x)\|_p^p\right)^{\frac{1}{p}}.\]
\end{notation}
Let $H_i$ denote the (singular or cellular) homology functor with coefficients in some fixed field $k$.  
In what follows (and throughout this paper), by a \emph{distance} we mean an \emph{extended psuedometric}; see \cref{Def: Extended_(Pseudo)Metric}.  
  
\begin{definition}[Stability Properties]
For $n\geq 1$ and $p\in [1,\infty]$, we say a distance $d$ on $n$-parameter persistence modules is 
\begin{itemize}
\item[(i)]\emph{stable} if for all topological spaces $T$, functions $f,g\colon T\to \R^n$, and $i\geq 0$, we have
\[d(H_i\mathcal{S}(f), H_i\mathcal{S}(g))\leq \|f-g\|_{\infty},\]
\item[(ii)] \emph{$p$-stable} if for any finite CW-complex $X$, monotone $f,g\colon \Cells(X)\to \R^n$, and $i\geq 0$, we have
\[d(H_i\mathcal{S}(f),H_i\mathcal{S}(g))\leq \|f-g\|_p,\]
\item[(iii)] \emph{$p$-stable across degrees with constant $c$} if for all $X$, $f$, and $g$ as in (ii), we have 
\[\|d(H_{*}\mathcal{S}(f),H_{*}\mathcal{S}(g))\|_p \leq c\|f-g\|_p,\]
where $d(H_*\mathcal{S}(f),H_*\mathcal{S}(g))$ denotes the sequence of non-negative real numbers
\[d(H_0\mathcal{S}(f),H_0\mathcal{S}(g)),\ d(H_1\mathcal{S}(f),H_1\mathcal{S}(g)),\ d(H_2\mathcal{S}(f),H_2\mathcal{S}(g)),\ \ldots\]
\end{itemize}
\end{definition}

\begin{theorem}[Stability and Universality of the Interleaving Distance \cite{lesnick2015theory}]\label{Thm:Stability_and_Universality_of_The_Interleaving_Distance}
For any $n\geq 1$,
\begin{itemize}
\item[(i)] $d_{\mcI}$ is stable on $n$-parameter persistence modules,
\item[(ii)] if the field of coefficients $k$ is prime and $d$ is any stable distance on $n$-parameter persistence modules, then $d\leq d_{\mcI}$.
\end{itemize}
\end{theorem}

In view of the isometry theorem, \cref{Thm:Stability_and_Universality_of_The_Interleaving_Distance}\,(i) generalizes the well-known stability theorem for sublevel persistent homology of $\R$-valued functions given in \cite{d2003optimal,cohen2007stability}.  In fact, \cref{Thm:Stability_and_Universality_of_The_Interleaving_Distance}\,(i) turns out to be trivial, but the proof of \cref{Thm:Stability_and_Universality_of_The_Interleaving_Distance}\,(ii) requires some work.

Recently, Skraba and Turner have established the following fundamental $\ell^p$-stability result for the 1-parameter persistent homology of filtered CW-complexes.

\begin{theorem}[$\ell^p$-Stability of 1-Parameter Persistent Homology \cite{skraba2020wasserstein}]\label{Thm:Skraba_Turner_Stability}
For all $p=[1,\infty]$, $d_{\W}^p$ is $p$-stable across degrees with constant $1$, and hence also $p$-stable.  
\end{theorem}

Note that in the special case of $\R$-valued monotone functions on finite CW-complexes, \cref{Thm:Stability_and_Universality_of_The_Interleaving_Distance}\,(i) coincides with the $p=\infty$ case of \cref{Thm:Skraba_Turner_Stability}.  In \cite{skraba2020wasserstein}, \cref{Thm:Skraba_Turner_Stability} is applied to obtain new stability results for sublevel filtrations of greyscale images, Vietoris-Rips complexes, and the persistent homology transform.

\begin{remark}
A 2010 paper of Cohen-Steiner et al. gives a different stability result for sublevel persistent homology using a Wasserstein distance on barcodes \cite{cohen2010lipschitz}.  This result concerns Lipschitz functions on triangulable, compact metric spaces, and uses the variant of $p$-Wasserstein distance on barcodes defined using the $\ell^\infty$-norm on $\R^2$, rather than the $\ell^p$-norm; see \cref{Rem:Waserstein_History_Variants}.
\end{remark}

\subsection{The $p$-Presentation and $p$-Matching Distances}
For $p\in [1,\infty]$, we introduce two generalizations of the $p$-Wasserstein distance $d_{\W}^p$ to finitely presented multiparameter persistence modules.  We call these the \emph{$p$-presentation distance} (or simply \emph{presentation distance}) and the \emph{$p$-matching distance}, and denote them by $d_{\mcI}^p$ and $d_\match^p$, respectively.  In the case $p=\infty$, they are respectively equal to the interleaving distance and matching distance. We show that several fundamental properties of the Wasserstein, interleaving, and matching distances extend to our new distances.  

A number of other $\ell^p$-type distances on multiparameter persistence modules have recently been proposed; we discuss these below, in \cref{Sec:Related_Work}.  However, our work is the first to consider a common generalization of the Wasserstein and interleaving distances, and also the first to consider a common generalization of the Wasserstein and matching distances.

We next give a brief description of $p$-presentation and $p$-matching distances; details can be found in later sections of the paper.  A \emph{presentation} of an $n$-parameter persistence module $M$ is a morphism of free modules $\gamma\colon F\to G$ whose cokernel is isomorphic to $M$.  If $M$ is finitely presented, then by choosing ordered bases of $F$ and $G$, we can represent $\gamma$ as a matrix with an $\R^n$-valued label attached to each row and each column; we call this a \emph{presentation matrix}.  
Given presentation matrices $P_M$ and $P_N$ for persistence modules $M$ and $N$ with the same underlying matrix, let $d^p(P_M,P_N)$ denote the $\ell^p$-distance between the sets of labels of $P_M$ and $P_N$.  Define $\hat d_{\mcI}^p(M,N)$ to be the infimum of $d^p(P_M,P_N)$ over all such choices of $P_M$ and $P_N$.  It turns out that $\hat d_{\mcI}^p$ does not satisfy the triangle inequality when $n\geq 2$ and $p\in [1,\infty)$.  We define $d_{\mcI}^p$ to be the largest distance that is bounded above by $\hat d_{\mcI}^p$ and satisfies the triangle inequality.
 
The matching distance $d_{\match}(M,N)$ is defined in terms of the bottleneck distance between 1-parameter affine restrictions of $M$ and $N$.  The definition of $d_\match^p$ amounts simply to replacing the bottleneck distance with the $p$-Wasserstein distance in the definition of $d_{\match}$.  

\subsection{Main Results}
The following theorem describes the relationship of our distances to each other, and to the other distances on persistence modules mentioned above.  
\begin{theorem}\label{Thm:Main_Thm_Intro}\mbox{}
On finitely presented multiparameter persistence modules,
\begin{itemize}
\item[(i)] $d_{\mcI}^\infty=d_{\mcI}$,
\item[(ii)] $d_{\match}^\infty=d_\match$,
\item[(iii)] $d_\match^p\leq d_{\mcI}^p$ for all $p\in [1,\infty]$.
\end{itemize}
On finitely presented 1-parameter persistence modules, 
\begin{itemize}
\item[(iv)] $d_{\mcI}^p=d_{\W}^p=d_{\match}^p$ for all $p\in [1,\infty]$.
\end{itemize}
\end{theorem}

The content of the theorem can be summarized by placing the expression $d_\match^p\leq d_{\mcI}^p$ in the lower right corner of \cref{Table:Generalizations_of_Bottleneck}, as we have done. 

We prove (i) and the first equality of (iv) in \cref{Sec:PresWDist}; (iii) is proven in \cref{Sec:W_matching}; and (ii) and the second equality of (iv) are immediate from the definitions. A proof of a generalization of (i) is implicitly given in \cite[Section 4]{lesnick2015theory} as part of the proof of \cref{Thm:Stability_and_Universality_of_The_Interleaving_Distance}\,(ii). That proof of (i) is somewhat technical; here we provide a more intuitive proof.  
Our proof of (iv) is similar to the proof of \cref{Thm:Skraba_Turner_Stability} given in \cite{skraba2020wasserstein}.

Together, \cref{Thm:Main_Thm_Intro}\,(i) and the $p=\infty$ case of \cref{Thm:Main_Thm_Intro}\,(iv) imply the isometry theorem (i.e., that $d_{\mcI}=d_{\W}^\infty$) for finitely presented 1-parameter persistence modules.  
In fact, our proofs of  \cref{Thm:Main_Thm_Intro}\,(i) and (iv) together amount to a novel proof of the isometry theorem.  While there already exist several good proofs of the isometry theorem which work in greater generality \cite{bauer2015induced,bauer2016persistence,bjerkevik2016stability,chazal2012structure}, we feel that the proof given here may be of interest, because it is simple, intuitive, and rests on established facts known to be useful elsewhere.  
Our proof is similar in some respects to existing proofs of algebraic stability and stability for $\R$-valued functions; specifically, \cite{cohen2006vines,skraba2020wasserstein} use a very similar interpolation strategy, and \cite{chazal2009proximity,chazal2012structure,cohen2007stability} use a somewhat similar one.

As an application of the first equality of  \cref{Thm:Main_Thm_Intro}\,(iv), we observe that on bipersistence modules, $d_\match^p$ can be efficiently approximated, up to arbitrarily small error, by a simple extension of the algorithm of \cite{kerber2020efficient} for approximating $d_{\match}$; see \cref{Sec:Computation} for details, including a precise runtime bound.  

In view of what is known about computation of $d_{\mcI}$ and $d_{\match}$, the following conjecture seems natural, and is left to future work:
\begin{conjecture}\label{Conj:Computational_Conjectures}\mbox{}
\begin{itemize}
\item[(i)] For fixed $n\geq 1$, the distance $d^p_{\match}$ on finitely presented $n$-parameter persistence modules is exactly computable in polynomial time. 
\item[(ii)] Computing $d_{\mcI}^p$ on finitely presented $2$-parameter persistence modules is NP-hard, for all $p\in [1,\infty]$.
\end{itemize}
\end{conjecture}

Our final result, which is perhaps the central result of the paper, concerns the stability and universality of the presentation distance on 1- and 2-parameter persistence modules.  The result extends Skraba and Turner's cellular $\ell^p$-stability result \cref{Thm:Skraba_Turner_Stability}, and is a close analogue of \cref{Thm:Stability_and_Universality_of_The_Interleaving_Distance}. 

In the statement, we adopt the convention that $\frac{1}{\infty}=0$.

\begin{theorem}[Stability and Universality of Presentation Distances]\label{Thm:Universality}
For any $p\in [1,\infty]$ and $n\in \{1,2\}$,
\begin{itemize}
\item[(i)] $d_{\mcI}^p$ is $p$-stable, and also $p$-stable across degrees with constant $n^\frac{1}{p}$, on finitely presented $n$-parameter persistence modules,
\item[(ii)] if the field of coefficients $k$ is prime and $d$ is any $p$-stable distance on finitely presented $n$-parameter persistence modules, then $d\leq d_{\mcI}^p$.
\end{itemize}
\end{theorem}

\begin{remarks}\mbox{}
\begin{enumerate}[(i)]
\item For both the cases $n=1$ and $n=2$, the constant $n^\frac{1}{p}$ appearing in \cref{Thm:Universality}\,(i) is tight; see \cref{Prop:Tightness}. 
\item In the case $n=1$, \cref{Thm:Universality}\,(ii) in fact holds for arbitrary fields $k$.
\item In the case $n=1$, \cref{Thm:Universality}\,(i) is an immediate consequence of Skraba and Turner's result \cref{Thm:Skraba_Turner_Stability} and \cref{Thm:Main_Thm_Intro}\,(iv), but can be proven more directly via an argument similar to the proof of \cref{Thm:Skraba_Turner_Stability}; see \cref{Sec:Proof_of_Stability}.
\item \cref{Thm:Main_Thm_Intro}\,(iii) implies that \cref{Thm:Universality}\,(i) also holds if $d_\mcI^p$ is replaced by $d_\match^p$.
\item Note that in the case $p=\infty$, \cref{Thm:Universality} is somewhat weaker than \cref{Thm:Stability_and_Universality_of_The_Interleaving_Distance}: \cref{Thm:Stability_and_Universality_of_The_Interleaving_Distance} holds for persistence modules with an arbitrary number of parameters, and does not require any finiteness assumptions.  In addition, \cref{Thm:Stability_and_Universality_of_The_Interleaving_Distance}\,(i) holds for the sublevel filtrations of arbitrary $\R^n$-valued functions, not only for monotone functions on cell complexes.  
\end{enumerate}
\end{remarks}

\subsection{Other Approaches to Defining $\ell^p$-Distances}\label{Sec:Related_Work}
The problem of generalizing the Wasserstein distance on barcodes was first studied by Bubenik, Scott, and Stanley \cite{bubenik2018wasserstein}, who consider the variant of $d_{\W}^p$ defined in terms of the $\ell^1$-norm on $\R^2$.  They introduce and study a generalization of this to diagrams of vector spaces indexed by an arbitrary small category, which includes the case of multiparameter persistence modules.  In this case, their distance, which we will denote as $D^p$, is not equal to either $d^p_{\match}$ or $d^p_\mcI$ for any $p$, and has rather different properties.  For example, $d_\mcI^p$ and $d^p_{\match}$ are always finite on finitely generated, free $\R^n$-indexed or $\Z^n$-indexed persistence modules with the same number of generators, but when $n\geq 2$, $D^p$ is infinite unless the modules are isomorphic.  Another qualitative difference is that for indecomposable persistence modules $M$ and $N$, $D^p(M,N)$ is in fact independent of $p$, but our distances generally are not.  
It is shown in \cite{bubenik2018wasserstein} that $D^p$ satisfies a universal property, though this is rather different than the universal property of \cref{Thm:Universality}\,(ii).

Skraba and Turner also consider generalizing the Wasserstein distance in \cite[Section 8]{skraba2020wasserstein}: They introduce an alternative definition $d_{\mathcal A}^p$ of the Wasserstein distance $d_\W^p$ on 1-parameter persistence modules and show that indeed  $d_{\mathcal A}^p=d_{\mathcal W}^p$ \cite[Theorems 8.25 and 8.26]{skraba2020wasserstein}.  This result generalizes the algebraic stability theorem.  The authors then study the question of generalizing $d_{\mathcal A}^p$ to other abelian categories, giving axioms on such a generalization which ensure that it satisfies the triangle inequality.  However, no generalization of $d_{\mathcal A}^p$ to multiparameter persistence modules is given in \cite{skraba2020wasserstein}, and the question of how to define one is mentioned as a direction for future work.  
Previous work by Scolamiero et al. introduced a similar (but not identical) axiomatic approach to defining distances on multiparameter persistence modules \cite[Definition 8.6]{scolamiero2017multidimensional}, though this work does not consider $\ell^p$-distances.  The definition of generalized $p$-Wasserstein distance appearing in \cite{bubenik2018wasserstein} also uses a construction very similar to \cite[Definition 8.6]{scolamiero2017multidimensional}. 

Recent work of Giunti et al. introduces a general framework for defining metrics on multiparameter persistence modules which extends each of 
the approaches discussed above and appears to yield novel $\ell^p$-distances \cite[Example~3.34, Remark 3.35, and Definition 4.9]{giunti2021amplitudes}. 

Another approach to defining $\ell^p$-type distances on multiparameter persistence modules is to vectorize the modules, i.e., to specify a map from modules into a vector space equipped with a $p$-norm; the induced metric on the vector space then pulls back to a distance on the modules.  Several such maps have been proposed and applied to data.  One very simple approach is to define the map in term of the Hilbert function, i.e., the dimension of the vector space at each index \cite{keller2018persistent,betancourt2018pseudo}.  Alternatively, when working with homology in all degrees, one can instead take the Euler characteristic at each index \cite{beltramo2021euler}.  As these approaches do not depend on the internal linear maps in the persistence modules, they are rather coarse.  Other vectorizations have been proposed that do depended on the internal maps: Vipond's \emph{multiparameter persistence landscapes} \cite{vipond2018multiparameter,vipond2021multiparameter} and Carri\'ere and Blumberg's \emph{multiparameter persistence images} \cite{carriere2020multiparameter} consider the internal maps of the module only along a fixed direction, whereas the kernel construction of Corbet et al. \cite{corbet2019kernel} considers the internal maps in all directions.  

Stability results have been given for each of the last three approaches, though these have some key limitations.  To elaborate, \cite{vipond2018multiparameter} shows that the multiparameter persistence landscapes are 1-Lipschitz stable with respect to the interleaving distance on modules and the $L^\infty$-distance on landscapes.  On the other hand, \cite[Section 7.2]{skraba2020wasserstein} observes that for all $p\in [1,\infty)$, 1-parameter persistence landscapes are not H\"older continuous with respect to the Wasserstein distance $\W_p$ on barcodes and the $L^p$-distance on landscapes.  A stability bound for the kernel construction of \cite{corbet2019kernel} is given with respect to the matching distance on modules and the $L^2$-distance, but the bound involves a constant which depends on the size of the input and can be quite large.  While the multiparameter persistence images of \cite{carriere2020multiparameter} are unstable in general, the authors give a partial stability result under special conditions \cite[Supplementary Material]{carriere2020multiparameter}.

A potential practical issue with some of the distances described above, e.g., the $L^p$-distance on Hilbert functions, is that they are often infinite on modules with unbounded support.  In an effort to address this issue, Miller and Thomas \cite{primary-distance} use primary decomposition to construct modified distances which are finite on a larger class of persistence modules, including all finitely generated modules.  A preliminary exposition of the ideas appears in Thomas's Ph.D. thesis
 \cite[\S4.3]{thomas2019invariants}.

\subsection{Applications}\label{Sec:Applications}
We describe three potential directions for applications of the distances introduced in this work.  A first natural direction is to use $d_{\mcI}^p$ to formulate stability and inference results for multiparameter persistent homology.  Some such results for the \emph{multicover bifiltration} will appear in a companion paper \cite{bjerkevikMulticover}, and are outlined in \cref{Sec:Multicover_Outline} below.

Second, we imagine that the $p$-matching distance for small $p$ could be used in practical applications of multiparameter persistent homology, in much the same way that the $p$-Wasserstein distance on barcodes has been used.  To elaborate, while $d_\match$ is a natural candidate for such practical use, it inherits some features of $d_{\W}^\infty$ that may be undesirable, namely, sensitivity to outlying algebraic structure and insensitivity to small features.  Thus, in analogy with the 1-parameter case, where $d_{\W}^1$ or $d_{\W}^2$  is often preferred over $d_{\W}^\infty$, one imagines that a variant of $d_\match$ generalizing $d_{\W}^p$ might perform better in applications than $d_\match$.  One example of a potential application along these lines is the \emph{virtual screening problem} in computational chemistry, i.e., the problem of identifying drug candidates from a large database of ligands.  The matching distance on 2-parameter persistence modules has been applied to this problem in \cite{keller2018persistent}.  (Persistent homology has also been applied to the problem in \cite{cang2018representability}.)  We hypothesize that $d_\match^1$ and $d_\match^2$ would outperform $d_\match$ in this application, though we do not explore this, or any other application to real data, in the present paper. 

A third potential direction, also not explored here, is to use the distance $d_{\mcI}^p$ (particularly, $d_{\mcI}^1$) to formulate stability results for vectorizations of multiparameter persistence modules.  As mentioned in \cref{Sec:Generalizations}, in the 1-parameter case there exist several such stability results for the $1$-Wasserstein distance on barcodes.  One might hope that some of these extend to the multiparameter setting.

\subsection{$\ell^p$-Stability of Multicover Persistence}\label{Sec:Multicover_Outline}
As mentioned above, in \cite{bjerkevikMulticover} we apply the distances $d_{\mcI}^p$ and $d_\match^p$ to study the stability and continuity of the \emph{multicover persistent homology} of point cloud data.  To motivate the results of this paper, we briefly describe this application. 

Multicover persistent homology is a natural 2-parameter extension of the standard union-of-balls construction of persistent homology \cite{sheehy2012multicover,chazal2011geometric}. It takes into account the number of times a point is covered by a ball, and thus is density-sensitive in a way that the standard construction is not.  Recent work has shown that multicover persistent homology is computable, at least for modestly sized data embedded in a low-dimensional Euclidean space \cite{corbet2021computing,edelsbrunner2020simple}.    

It was also recently shown that multicover persistent homology satisfies a 1-Lipschitz stability property closely analogous to the usual stability results for persistent homology \cite{blumberg2020stability}.  The property is stated in terms of the Prohorov distance $\pi$, a classical distance on probability measures, and the interleaving distance $d_{\mcI}$.   The result tells us in particular that, unlike the 1-parameter union-of-balls persistence, multicover persistence is robust to outliers.  
However, because this stability result is formulated in terms of $d_{\mcI}$, it is rather coarse, in a sense: According to the characterization of $d_{\mcI}$ provided by \cref{Thm:Main_Thm_Intro}\,(i), $d_{\mcI}$ can be defined in terms of a sup-norm on the grades of generators and relations, and thus is sensitive only to the largest perturbation of such  grades, not to the number of perturbations. Thus, the stability result of  \cite{blumberg2020stability} says nothing about how many small-scale algebraic changes to the persistent homology can result from a small change to the data.
In fact, a computational example in \cite{corbet2021computing} indicates that the addition of even a few random outliers to a data set can generate quite a lot of small-scale algebraic noise in the multicover persistent homology.  This raises the question of whether, by using different metrics on persistence modules, we can develop a more nuanced picture of the stability of multicover persistence, in which such noise is visible.

In answer to this, we show in \cite{bjerkevikMulticover} that when formulated using our distances $d_{\mcI}^p$ and $d_\match^p$ for $p<\infty$, the stability story for multicover persistent homology looks rather different than it does for $d_{\mcI}=d_{\mcI}^\infty$.  Specially, we show that for each $p\in [1,\infty)$ and $m\geq 2p$, multicover persistent homology of points in $\R^m$ is discontinuous with respect to $\pi$ and $d_\mcI^p$.  However, if we restrict to point clouds of bounded cardinality, then multicover persistent homology is Lipschitz continuous for all $p\in [1,\infty)$.  The same results also hold using $d^p_{\match}$ in place of $d_\mcI^p$, or using the Wasserstein distance on probability measures in place of $\pi$.

\subsection*{Acknowledgements}
We thank Andrew Blumberg for valuable discussions about matters related to this work and the companion work \cite{bjerkevikMulticover}.  We also thank Barbara Giunti and Nina Otter for sharing insights into related work on metrics for generalized persistence modules, and we thank Ezra Miller for helpful feedback on a draft of this paper. The first author is supported by the Austrian Science Fund (FWF) grant number P 33765-N.

\section{Preliminaries}
\subsection{Persistence Modules}
A category $C$ is said to be \emph{thin} if for every pair of objects $a,b$ in $C$, $\Hom(a,b)$ contains at most one morphism.  Any poset $(X,\leq)$ can be regarded as a thin category with object set $X$, by taking $\Hom(a,b)$ to be non-empty if and only if $a\leq b$. 

For $C$ a small, thin category, we define a \emph{$C$-persistence module} to be a functor $M\colon C\to \Vect$, where $\Vect$ denotes the category of $k$-vector spaces over some fixed field $k$. Thus, $M$ consists of a vector space $M_a$ for each object $a$ of $C$, and a linear map $M_{a\subto b}$ for each morphism $a\to b$ in $C$, such that $M_{a\subto a}=\id_{M_a}$ and $M_{b\subto c}\circ M_{a\subto b}=M_{a\subto c}$. When the category $C$ is understood, we often omit $C$ and simply call $M$ a ``persistence module'' or even just a ``module".  If $\dim(M_a)<\infty$ for all $a$, then we say that $M$ is \emph{pointwise finite dimensional}, or \emph{\pfd}

A morphism $\gamma \colon M\to N$ of $C$-persistence modules is a natural transformation, i.e., a choice of a linear map $\gamma_a \colon M_a\to N_a$ for each object $a$ of $C$ such that for each morphism $a\to b$ in $C$, \[N_{a\subto b}\circ \gamma_a=\gamma_b\circ M_{a\subto b}.\]  With this definition of morphism, the $C$-persistence modules form an abelian category $\Vect^C$.  Hence, many standard constructions of  abstract algebra 
are well-defined in $\Vect^C$, e.g., submodules, quotients, kernels, images, and direct sums.  The definitions of these are given objectwise.  For example, the direct sum $M\oplus N$ of persistence modules $M$ and $N$ is given by 
\[(M\oplus N)_{a}=M_a\oplus N_a\qquad (M\oplus N)_{a\subto b}=M_{a\subto b}\oplus N_{a\subto b}.\]
Similarly, the kernel of a morphism $\gamma \colon M\to N$ of persistence modules, denoted $\ker \gamma$, is given by 
\[(\ker \gamma)_a=\ker \gamma_a,\]
with the internal maps taken to be the restriction of those in $M$.

In this paper, we are interested primarily in $\R^n$-persistence modules, where $\R^n$ is given the partial order of \cref{Sec:Stability_Intro}.  We sometimes call these \emph{$n$-parameter persistence modules}.  Interleavings, defined in \cref{Sec:Interleavings} below, are a second example of persistence modules that we will want to consider.

\begin{definition}\label{Def:Interval_Module}
An \emph{interval} in $\R$ is a non-empty connected subset.  For $I\subset \R$ an interval, define the \emph{interval module} $k^I$ to be the 1-parameter persistence module such that 
\begin{align*}
k^I_a&=
\begin{cases}
k &{\textup{if }} a\in I, \\
0 &{\textup{ otherwise}.}
\end{cases}
& k^I_{a\subto b}=
\begin{cases}
\mathrm{Id}_k &{\textup{if }} a\leq b\in I,\\
0 &{\textup{ otherwise}.}
\end{cases}
\end{align*}
\end{definition}

\begin{theorem}[Structure of 1-Parameter Persistence Modules \cite{crawley2012decomposition}]\label{Thm:ordinary_Structure}
If $M$ is a \pfd $\R$-persistence module, then there exists a unique multiset $\B{}_M$ of intervals in $\R$ such that 
\[M\cong {\oplus_{I\in \B{}_M}} k^I.\]
We call $\B{}_M$ the \emph{barcode of $M$}.
\end{theorem}
In general, we call a multiset $\B{}$ of intervals in $\R$ a barcode.

\subsection{Free Modules and Presentations}
\label{Sec:Free_Mods}
Free $\R^n$-persistence modules arise frequently in TDA.  Their definition in fact extends immediately to a definition of free $C$-persistence modules for any small category $C$; for example, see \cite[Section 4]{tchernev2015modules}.  We will only need to consider the case where the indexing category is a poset, so we give the definition just in this case.  

For $X$ a poset and $x\in X$, let $Q^x$ denote the $X$-persistence module given by 
\begin{align*}
Q^x_a &=
\begin{cases}
k &\text{if } x \leq a, \\
0 &\text{otherwise,}
\end{cases}
& Q^x_{a,b}=
\begin{cases}
\mathrm{Id}_k &{\text{if } x\leq a},\\
0 &{\text{otherwise}}.
\end{cases}
\end{align*}

\begin{definition}
We say an $X$-persistence module is \emph{free} if there exists a multiset $\mathcal Y$ of elements in $X$ such that $F\cong \oplus_{x\in \mathcal Y}\, Q^x$.
\end{definition}

For $a\in X$ and $v\in M_a$, we write $\gr(v)=a$. We say that $S\subset \bigcup_{a\in X} M_a$ is a \emph{set of generators for $M$} if for any $v\in \bigcup_{a\in X} M_a$, \[v=\sum_{i=1}^m c_i M_{\gr(v_i)\subto \gr(v)}(v_i)\] for some $v_1,v_2,\ldots, v_m\in S$ and scalars $c_1,\ldots, c_m\in k$.  

\begin{definition}
A \emph{basis} of a free $X$-persistence module $F$ is a minimal set of generators for $F$. 
\end{definition}

\begin{remark}
Alternatively, one can give equivalent definitions of free modules and bases in terms of a universal property \cite{carlsson2009theory,tchernev2015modules}.
\end{remark}

Given a free $X$-persistence module $F$ and a basis $B$ for $F$, write $B^a:=\{b\in B \mid \gr(b)=a\}$. 

\begin{proposition}\label{prop:Basis_Grades}
For any free $X$-persistence module $F$ and $a\in X$, the cardinality of $B^a$ is the same for all bases $B$ of $F$.
\end{proposition}

\begin{proof}
Define $V\subset F_a$ by \[V=\sum_{a'<a} \im(F_{a'\subto a}).\]  It is easily checked that $B^a$ descends to a basis for the vector space $F_a/V$.  Since all bases of a vector space have the same cardinality, the result follows.
\end{proof}

\paragraph{Matrix Representation of Morphisms of Free Modules}
Let $B=(b_1,\ldots,b_{|B|})$ be an ordered basis of a finitely generated free $X$-persistence module $F$, and let $b_i$ denote its $i^{\mathrm{th}}$ element.  For $a\in X$, we can represent a vector $v\in F_a$ with respect to $B$ as a vector $[v]^B\in k^{|B|}$; we take $[v]^B$ to be the unique vector such that $[v]_i^B=0$ if $\gr(b_i)\not\leq a$ and \[v=\sum_{i: \gr(b_i)\leq a} [v]^B_i F_{\gr(b_i)\subto a}(b_i).\] Thus, $[v]^B$ records the field coefficients in the linear combination of $B$ giving $v$.   

Along similar lines, for $B'=(b'_1,\ldots,b'_{|B'|})$ a finite ordered basis of a free persistence module $F'$, we represent a morphism $\gamma\colon F\to F'$ as a matrix $[\gamma]^{B',B}$ with coefficients in the field $k$, with each row and column labeled by an element of $X$, as follows:
\begin{itemize}
\item The $j^{\mathrm{th}}$ column of $[\gamma]^{B',B}$ is $[\gamma(b_j)]^{B'}$.
\item The label of the $j^{\mathrm{th}}$ column is $\gr(b_j)$,
\item The label of the $i^{\mathrm{th}}$ row is $\gr(b'_i)$.
\end{itemize}
Where no confusion is likely, we sometimes write $[\gamma]^{B',B}$ simply as $[\gamma]$.    In the literature on multigraded commutative algebra, $[\gamma]$ is typically called a \emph{monomial matrix} \cite{miller2000alexander,miller2004combinatorial}, though we will not use this term.

\paragraph{Presentations}

A \emph{presentation} of an $X$-persistence module $M$ is a morphism of free $X$-persistence modules $\gamma\colon F\to G$ such that $\coker \gamma \cong M$.  In view of the above, if $F$ and $G$ are finitely generated, then we can represent $\gamma$ with respect to ordered bases of 
$F$ and $G$ as a matrix $[\gamma]$ with coefficients in $k$, together with an $X$-valued label for each row and each column.  We call the labeled matrix $[\gamma]$ a \emph{presentation matrix} for $M$, or simply a presentation (abusing terminology slightly).  It is easily checked that $[\gamma]$ encodes $\gamma$ up to natural isomorphism; that is, given $[\gamma]$, we can construct a morphism of free modules $\gamma'\colon F'\to G'$ such that there exist isomorphisms $\alpha\colon F\to F'$ and $\beta \colon G\to G'$ satisfying $\beta\circ \gamma=\gamma'\circ \alpha$.  If there exists a presentation $\gamma \colon F\to G$ of $M$ with $F$ and $G$ finitely generated, then we say $M$ is \emph{finitely presented}.

\begin{example}
A module $F$ is free if and only if $0\to F$ is a presentation of $F$.
Choosing a basis of $F$, we get a presentation matrix that has no columns and therefore contains no data except a vector of $X$-valued row labels. 
This vector records the grades of the basis elements.
\end{example}

\begin{definition}\label{Def:tame_barcode}
We say that a barcode $\B{}$ is \emph{finitely presented} if $\B{}$ consists of finitely many intervals, each of the form $[a,b)$ with $a<b\in \R\cup\{\infty\}$.  
\end{definition}

\begin{remark}
It is easy to check that if $M$ is a finitely presented $\R$-persistence module, then $\B{}_M$ is finitely presented.
\end{remark}

\subsection{Interleavings}\label{Sec:Interleavings}
Fix $n\geq 1$, and for $\delta\in [0,\infty)$, let $\vec \delta=(\delta,\delta,\ldots,\delta)\in \R^n$.  Define the \emph{$\delta$-interleaving category} $\I^\delta$ to be the thin category with object set $\R^n\times \{0,1\}$ and a morphism $(a,i)\to (b,j)$ if and only if either  
\begin{enumerate}
\item $a+\vec\delta \leq b$, or
\item $i=j$ and $a\leq b$.
\end{enumerate}
For $i\in \{0,1\}$, let ${\mathcal J}^i\colon\R^n\to \I^\delta$ be the functor mapping $a\in \R^n$ to $(a,i)$.  Thus, ${\mathcal J}^i$ is the restriction of $\I^\delta$ to one of the two copies of $\R^n$.

\begin{definition}
For $\delta\in [0,\infty)$, a \emph{$\delta$-interleaving} between $\R^n$-persistence modules $M$ and $N$ is a functor
\[
Z\colon \I^\delta\to \Vect
\]
such that $Z\circ {\mathcal J}^0=M$ and $Z\circ {\mathcal J}^1=N$.
\end{definition}

\begin{remark}
Equivalently, a $\delta$-interleaving can be defined as a certain pair of natural transformations, as for example in \cite{bubenik2015metrics}, and this viewpoint is a useful one.  However, for the purposes of this paper,  the definition we have given here is more convenient.
\end{remark}

We define the interleaving distance $d_{\mcI}$ between functors $M,N\colon \R^n\to \Vect$ by \[d_{\mcI}(M,N)=\inf\, \{\delta\mid \textup{There exists a $\delta$-interleaving between $M$ and $N$}\}.\]

\begin{definition}[Extended Metrics and Pseudometrics]\label{Def: Extended_(Pseudo)Metric}
Given a set $X$ in some Grothendieck universe, an \emph{extended pseudometric on $X$} is a function $d \colon X\times X\to [0,\infty]$ such that for all $x,y,z\in X$
\begin{itemize}
\item $d(x,x)=0$,
\item $d(x,y)=d(y,x)$,
\item $d(x,z)\leq d(x,y)+d(y,z)$.
\end{itemize}
If in addition $d(x,y)>0$ whenever $x \ne y$, then $d$ is called an \emph{extended metric}.  (By convention, ``extended" is understood in this context to mean that $d$ can take value $\infty$.)
\end{definition}
As noted in the introduction, in this paper a \emph{distance} is understood to be an extended pseudometric.

\begin{remark}\label{Rem:Interleaving_Dist_Is_Metric} $d_{\mcI}$ is an extended pseudometric on $n$-parameter persistence modules.  Moreover, it is shown in \cite{lesnick2015theory} that $d_{\mcI}$ descends to an extended metric on isomorphism classes of finitely presented modules.
\end{remark}

\subsection{The Wasserstein Distance}\label{Sec:Wasserstein_Distance}\label{Sec:Wasserstein}
We next define the Wasserstein distance on barcodes.  To keep notation simple, we give the definition only for finitely presented barcodes (\cref{Def:tame_barcode}), as this is the only case we use in this paper, but the definition extends without difficulty to arbitrary barcodes.

A \emph{matching} between sets $S$ and $T$ is a bijection 
$\sigma\colon S'\to T'$, 
where $S'\subset S$ and $T'\subset T$.  Formally, we may regard $\sigma$ as a subset of $S\times T$, where $(s,t)\in \sigma$ if and only if $\sigma(s)=t$.  The definition of a matching extends readily to multisets.

In what follows, we will freely use the standard conventions for arithmetic over the extended reals $\R\cup\{-\infty, \infty\}$.  
In addition, we adopt the convention that $\infty-\infty=0$.  
  
Recall from \cref{Def:tame_barcode} that each interval in a finitely presented barcode is of the form $[a_1,a_2)$ for $a_1<a_2\in \R\cup \{\infty\}$; we will identify this interval with the point $a=(a_1,a_2)\in \R\times (\R\cup \{\infty\})$.  
We let 
\[m(a)=\left(\frac{a_1+a_2}{2},\frac{a_1+a_2}{2}\right)\in (\R\cup\{\infty\})^2.\]  

\begin{definition}[Wasserstein Distance on Barcodes]
Let $\B{}$ and $\C$ be finitely presented barcodes.  Given $p\in [1,\infty)$ and a matching $\sigma\colon\B{} \to \C$, we define 
\[\cost(\sigma,p)=\left(\sum_{(a,b)\in \sigma}\|a-b\|^p_p\ +\ \sum_{a\in \B{}\cup \C \textup{ unmatched by $\sigma$}} \|a-m(a)\|^p_p\right)^{\frac{1}{p}}.\]
Similarly, we define 
\[\cost(\sigma,\infty)=\max\left(\max_{(a,b)\in \sigma}\|a-b\|_\infty,\ \max_{a\in \B{}\cup \C \textup{ unmatched by $\sigma$}} \|a-m(a)\|_\infty\right).\]
For $p\in [1,\infty]$, the \emph{$p$-Wasserstein distance between $\B{}$ and $\C$}, denoted $d_{\W}^p(\B{},\C)$, is defined as 
\[d_{\W}^p(\B{},\C)=\min_{\sigma\colon\B{}\to \C\textup { a matching}} \cost(\sigma,p).\]
\end{definition}

It is easily checked that \[d_{\W}^\infty(\B{},\C)=\lim_{p\to\infty} d_{\W}^p(\B{},\C).\]

\begin{remark}
As every \pfd $\R$-persistence module has a well-defined barcode, we can regard $d_{\W}^p$ as a distance on such modules.
\end{remark}

\begin{remark}\label{Rem:Waserstein_History_Variants}
To the best of our knowledge, this version of the Wasserstein distance first appeared in the TDA literature for arbitrary $p\in [1,\infty]$  in \cite{robinson2013hypothesis}.  However, the $p=1$ case was considered earlier \cite{carlsson2004persistence}, as was the $p=\infty$ case (i.e., the bottleneck distance)  \cite{d2003optimal,brucale2002size}.  As mentioned in the introduction, Cohen-Steiner et al. \cite{cohen2010lipschitz} considered the variant of $d_{\W}^p$ where in the definition of $\cost(\sigma,p)$, the $\ell^p$-norm on $\R^2$ is replaced by the $\ell^\infty$-norm.  Similarly, Bubenik et al. \cite{bubenik2018wasserstein} replaced the $\ell^p$-norm on $\R^2$ with the $\ell^1$-norm.
\end{remark}

We are now ready for a formal statement of the isometry theorem:
\begin{theorem}[Isometry Theorem \cite{lesnick2015theory,chazal2012structure,bubenik2014categorification}]\label{Thm:Isometry}
For all \pfd $\R$-persistence modules $M$ and $N$, we have \[d_{\mcI}(M,N)=d_{\W}^\infty(M,N).\] 
\end{theorem}
In fact, it was shown by Chazal et al. that \cref{Thm:Isometry} holds for a more general class of modules, the \emph{$q$-tame} modules \cite{chazal2012structure,bauer2015induced}, though defining barcodes in this generality requires some care \cite{chazal2016observable}.

\subsection{The Matching Distance}\label{Sec:Matching_Distance}

Let $M$ be an $\R^n$-persistence module and $l\colon\R\to \R^n$ be a parameterized line of the form $l(t) = t\vec{v}+\vec w$, where $0\ne \vec v\in [0,\infty)^n$. We say $l$ is \emph{admissible} if the smallest coordinate of $\vec v$ is 1.
Let $M^l$ denote the $\R$-persistence module $M\circ l$, i.e., $M_a^l= M_{l(a)}$ and $M_{a \subto b}^l = M_{l(a) \subto l(b)}$.   

\begin{definition} The \emph{matching distance} between \pfd $\R^n$-persistence modules $M$ and $N$ is defined by 
\begin{align*}
d_\match(M,N) = \sup_{l}\, d_{\W}^\infty(M^l,N^l),
\end{align*}
where $l\colon\R\to \R^n$ ranges over all admissible lines. 
\end{definition}

\begin{proposition}[\cite{landi2018rank,cerri2013betti}]\label{Prop:Matching_Distance_Lower_Bound}
For all \pfd $\R^n$-persistence modules $M$ and $N$, $d_\match(M,N)\leq d_{\mcI}(M,N)$.
\end{proposition}

\begin{proof}
It is easily checked that a $\delta$-interleaving between $M$ and $N$ induces a $\delta$-interleaving between $M^l$ and $N^l$, for any admissible line $l$.  The result now follows from the algebraic stability theorem.
\end{proof}

\subsection{Left Kan Extensions along Grid Functions}
The material of this subsection will be used in the proofs of \cref{Thm:Main_Thm_Intro}\,(i) and \cref{Thm:Universality}\,(i).

\begin{definition}[\cite{lesnick2015interactive,botnan2018algebraic}]
Let $T=\Z$ or $T=\R$.  A \emph{grid function} is a function $\G\colon T^n\to \R^n$ of the form \[\G(a_1,\ldots,a_n)=(\G_1(a_1),\ldots,\G_n(a_n)),\] where each $\G_i\colon T\to \R$ is a functor (i.e. order-preserving function) such that  \[\lim_{a\to -\infty} \G_i(a)=-\infty\quad \textup{ and }\quad\lim_{a\to \infty} \G_i(a)=\infty.\]
If $T=\R$, then we also require that each $\G_i$ is left continuous.
\end{definition}
Note that in the above definition, we do not assume the $\G_i$ to be injective.  For $\G$ a grid function, define $\G^{-1}\colon\R^n\to T^n$ by \[\G^{-1}(a_1,\dots,a_n)=(\G^{-1}_1(a_1),\cdots,\G^{-1}_n(a_n)),\]
where \[\G^{-1}_i(a)=\max\,\{t \in T\mid \G_i(t)\leq a\}.\]   
The left continuity assumption in the definition of a grid function ensures that these maxima exist.

Left Kan extensions along grid functions have an especially simple description.  It will be convenient to take this description as a definition:

\begin{definition}[Left Kan Extension along a Grid Function]\label{Def:Concrete_LKan_Description}
For any grid function $\G\colon T^{n}\to \R^n$ and category $\C$, define a functor $\Lan_\G\colon\C^{T^n}\to \C^{\R^n}$ by 
\begin{align*}
\Lan_\G(F)_{a}&=F_{\G^{-1}(a)},\\
 \Lan_\G(F)_{a\subto b}&=F_{\G^{-1}(a)\subto \G^{-1}(b)},
 \end{align*}
where $a\leq b\in \R^n$. Similarly, for any natural transformation $\gamma$ of $T^n$-persistence modules, define $\Lan_\G(\gamma)$ by
\[\Lan_\G(\gamma)_{a}=\gamma_{\G^{-1}(a)}.\] 
It is readily checked that $\Lan_\G$ is indeed a functor.  This functor (or any functor naturally isomorphic to it) is called the \emph{left Kan extension along $\G$}.  
\end{definition}

\begin{remark}
It is straightforward to check that the functor $\Lan_\G$ of \cref{Def:Concrete_LKan_Description} is indeed the standard left Kan extension functor along $\G$, as defined, e.g., in \cite[Proposition 6.1.5]{riehl2017category}.  However, for our purposes, it will be sufficient to work directly with \cref{Def:Concrete_LKan_Description}, avoiding appeals to the standard definition.
\end{remark}

We now fix a grid function $\G$.  

\begin{lemma}\label{Lem:Exact_Extension}
If $\C=\Vect$, then $\Lan_\G$ is exact, i.e., given a short exact sequence of $T^n$-persistence modules \[0\to M\xrightarrow{\gamma} N\xrightarrow{\kappa} Q\to 0,\]
applying $\Lan_\G$ to the sequence yields an exact sequence:
\[0\to \Lan_{\G}(M)\xrightarrow{\Lan_{\G}(\gamma)} \Lan_{\G}(N)\xrightarrow{\Lan_{\G}(\kappa)} \Lan_{\G}(Q)\to 0.\]
\end{lemma}
\begin{proof}
Exactness of a sequence of persistence modules is an objectwise property, so this immediate from \cref{Def:Concrete_LKan_Description}. 
\end{proof}

\begin{notation}\label{Notation:Kan_Extended_Basis}
For $B$ a basis of a free $T^n$-persistence module $F$, let 
\[\Lan_\G(B)=\{b\in \Lan_\G(F)_{\G(\gr(b))}\mid b\in B\}.\]
To see that the above definition in fact makes sense, note that if $b\in B$, then by \cref{Def:Concrete_LKan_Description}, $b\in F_{\gr(b)}=\Lan_\G(F)_{\G(\gr(b))}$.
\end{notation}

\begin{proposition}\label{Prop:Free_Modules_and_Extension}
For any free $T^n$-persistence module $F$, 
\begin{itemize}
\item[(i)] $\Lan_\G(F)$ is free.
\item[(ii)] $\Lan_\G(B)$ is a basis of $\Lan_\G(F)$.
\item[(iii)] Given a morphism $\gamma\colon F\to G$ of free $T^n$-persistence modules and ordered bases $B$, $B'$ of $F$ and $G$,  the unlabeled matrices underlying $[\gamma]^{B',B}$ and $[\Lan_{\G}(\gamma)]^{\Lan_{\G}(B'),\Lan_{\G}(B)}$ are the same.
\end{itemize}
\end{proposition}

\begin{proof}
To prove (i), first note that $\Lan_\G$ preserves direct sums, because direct sums are defined objectwise.  (More generally, left Kan extensions preserve coproducts, because they are left adjoints \cite[Proposition 6.1.5]{riehl2017category}, and left adjoints preserve colimits \cite[Theorem 4.5.3]{riehl2017category}.)  Moreover, it is an easy consequence of \cref{Def:Concrete_LKan_Description} that $\Lan_\G(Q^x)=Q^{\G(x)}$ for all $x\in T^n$.  Item (i) now follows.  Items (ii) and (iii) follow easily from \cref{Def:Concrete_LKan_Description}.
\end{proof}

\begin{lemma}\label{Lem:Fin_Gen_LKE}
Any morphism of $\gamma:M\to N$ of finitely presented $\R^n$-persistence modules is the left Kan extension of a morphism $\gamma':M'\to N'$ of finitely generated $\Z^n$-persistence modules along a grid function $\G\colon \Z^n \to \R^n$. 
\end{lemma}

\begin{proof}
Let $P_M$ and $P_N$ be presentation matrices for $M$ and $N$, and let $\G$ be any injective grid function such that $\im \G$ contains all row and column labels of both $P_M$ and $P_N$.  Let 
 \[M'=M\circ \G,\quad N'=N\circ \G,\quad\textup{and}\quad\gamma'=\gamma\circ \G.\]  It is easily checked that $M'$ and $N'$ are finitely generated, and that $\gamma$ is naturally isomorphic to $\Lan_{\G}(\gamma')$.
\end{proof}

\section{The $p$-Presentation Distance}\label{Sec:PresWDist}
\subsection{Definition and Basic Properties}
If $P$ is an $r\times c$ presentation matrix of an $\R^n$-persistence module, then we may regard the set of $\R^n$-labels for the rows and columns as a function \[\labels(P)\colon\{1,2,\ldots,r+c\}\to \R^n\] by taking all rows to be ordered before all columns.  Equivalently, we may view $\labels(P)$ as a vector in $(\R^{n})^{(r+c)}$, i.e., a vector of length $(r+c)$ where each entry lies in $\R^n$.

Let $P_{M,N}$ denote the set of all pairs $(P_M,P_N)$ of finite presentations of $M$ and $N$, respectively, with the same underlying matrix.  Recalling \cref{Not:function_p_norm}, for $p\in [1,\infty]$ and  $(P_M,P_N)\in P_{M,N}$, we define \[d^p(P_M,P_N)=\|\labels(P_M)-\labels(P_N)\|_p.\]
For $M$ and $N$ finitely presented, we define \[\hat d_{\mcI}^p(M,N)=\inf_{(P_M,P_N)\in P_{M,N}} d^p(P_M,P_N).\]

By \cref{Thm:Main_Thm_Intro}\,(i) and \cref{Rem:Interleaving_Dist_Is_Metric},  $d_{\mcI}^\infty=d_{\mcI}$ descends to a metric on isomorphism classes of persistence modules.  However, the following example shows that for $p\in [1,\infty)$, $\hat d^p_\mcI$ does not satisfy the triangle inequality.

\begin{example}
\label{Ex:Triangle_eq_fails}
Fix $p\in [1,\infty)$, and let $M$ be the $\R^2$-persistence module given by
$$
M_a=
\begin{cases}
k &\text {if }a\geq (0,-1)\text{ or } a\geq (-1,0),\\
0 &\text{ otherwise},
\end{cases}
$$
with $M_{a\subto b}$ the identity whenever $M_a=k$. Recall from \cref{Sec:Free_Mods} that for $x\in \R^2$, $Q^x$ denotes the free module with one generator at $x$. We claim that for all $r\geq 0$,
\begin{align*}
\hat d_{\mcI}^p(M,Q^{(r,r)}) &\geq 4^\frac{1}{p}r,\\
\hat d_{\mcI}^p(M,Q^{(0,0)}) &\leq 2^\frac{1}{p},\\
\hat d_{\mcI}^p(Q^{(0,0)},Q^{(r,r)}) &\leq 2^\frac{1}{p}r.
\end{align*}
Thus, \[\hat d_{\mcI}^p(M,Q^{(r,r)})>\hat d_{\mcI}^p(M,Q^{(0,0)})+\hat d_{\mcI}^p(Q^{(0,0)},Q^{(r,r)})\] for $r$ large enough, violating the triangle inequality.

To obtain the lower bound for $d_{\mcI}^p(M,Q^{(r,r)})$, let $\gamma\colon F\to G$ be a presentation of $M$, and let $[\gamma]$ be its matrix representation with respect to a choice of ordered bases for $F$ and $G$.  Write the basis of $G$ as $(g_1,\dots,g_k)$.  Since $M_{(0,-1)}=k$ and $M_a=0$ for all $a<(0,-1)$, $G$ must have a basis element $g_i$ at $(0,-1)$, meaning that the $i^\textrm{th}$ row label of $[\gamma]$ is $(0,-1)$. Moreover, since $M_{(0,-1)\subto a}$ is an isomorphism for all $a\geq (0,-1)$, we may choose $g_i$ so that 
\[G_{(0,-1)\subto a}(g_i)\not \in \im \gamma_a\]
for all $a\geq (0,-1)$. It then follows that the $i^{\mathrm{th}}$ standard basis vector of $\R^{k}$ does not lie in the column space of $[\gamma]$. Similarly, $G$ also has a basis element $g_j$ at $(-1,0)$ such that the $j^{\mathrm{th}}$ standard basis vector of $\R^{k}$ is not in the column space of $[\gamma]$.

Now let $\gamma'\colon F'\to G'$ be a presentation of $Q^{(r,r)}$, and let $[\gamma']$ be its matrix representation with respect to a choice of ordered bases for $F'$ and $G'$ such that $[\gamma]$ and $[\gamma']$ have the same underlying matrix.  Let $l_i$ and $l_j$ denote the respective labels for rows $i$ and $j$ of $[\gamma']$.  Since the $i^{\mathrm{th}}$ and $j^{\mathrm{th}}$ standard basis vectors are not in the column space of $[\gamma']$, we must have $Q^{(r,r)}_{l_i}\ne 0$ and $Q^{(r,r)}_{l_j}\ne 0$, which implies that $l_i,l_j\geq (r,r)$.  Therefore, 
 \begin{align*}
d^p([\gamma],[\gamma'])&\geq \|((r,r)-(0,-1),(r,r)-(-1,0))\|_p\\
&= (r^p+(r+1)^p+(r+1)^p+r^p)^\frac{1}{p}\\
&\geq 4^\frac{1}{p}r.
\end{align*}
Since we have chosen the presentation of $M$ arbitrarily, it follows that $\hat d_{\mcI}^p(M,Q^{(r,r)}) \geq 4^\frac{1}{p}r$.

Next, we show that $\hat d_{\mcI}^p(M,Q^{(0,0)}) \leq 2^\frac{1}{p}$. $M$ and $Q^{(0,0)}$ can both be presented with the matrix 
$[1;-1]$. To present $M$, we take the row labels to be $(0,-1)$ and $(-1,0)$, and the column label to be $(0,0)$. To present $Q^{(0,0)}$, we let all the labels be $(0,0)$.  
The $d^p$-distance between these presentations is \[\|((0,0)-(0,-1),(0,0)-(-1,0),(0,0)-(0,0))\|_p = (1^p+1^p)^\frac{1}{p} = 2^\frac{1}{p},\] giving $\hat d_{\mcI}^p(M,Q^{(0,0)}) \leq 2^\frac{1}{p}$.

To show the last inequality, we use the obvious presentations $0\to Q^{(0,0)}$ and $0\to Q^{(r,r)}$. The $d^p$-distance between these presentations is $\|(r,r)-(0,0)\|_p=2^\frac{1}{p}r$, so $\hat d_{\mcI}^p(Q^{(0,0)},Q^{(r,r)}) \leq 2^\frac{1}{p}r$.

\end{example}

Though $\hat d^p_\mcI$ does not satisfy the triangle inequality, it generates an extended metric $d^p_\mcI$ on isomorphism classes of persistence modules:

\begin{definition}[Presentation Distance]\label{Def:Pres_Distance}
For $\R^n$-persistence modules $M$ and $N$ and $p\in [1,\infty]$, define \[d^p_\mcI(M,N)=
\inf\, \sum_{i=0}^{l-1} \hat d^p_\mcI(Q_i,Q_{i+1}),\] where the infimum is  taken over all finite sequences $Q_0,Q_1,\ldots,Q_l$ of finitely presented persistence modules with $Q_0=M$ and $Q_l=N$. We call $d_\mcI^p$ the \emph{$p$-presentation distance}.
\end{definition}

\begin{proposition}\label{Prop:Largest_Lower_Bound}\mbox{}
\begin{itemize}
\item[(i)] $d_{\mcI}^p$ is a distance (i.e., extended pseudometric) on finitely presented $\R^n$-persistence modules.
\item[(ii)] $d_{\mcI}^p$ is the largest distance bounded above by $\hat d_{\mcI}^p$; that is, if $d$ is a distance on $\R^n$-persistence modules and $d(M,N)\leq \hat d_{\mcI}^p(M,N)$ for all $M,N$, then $d\leq d_{\mcI}^p$.
\end{itemize}
\end{proposition}

\begin{proof}
It is clear that $d_{\mcI}^p$ satisfies all the properties of an extended pseudometric besides the triangle inequality.  To check that $d_{\mcI}^p$ also satisfies the triangle inequality, let $M,$ $M'$, and $N$ be modules.  For all $\delta>d_\mcI^p(M,M')$ and $\delta'>d_\mcI^p(M',N)$, there exist modules $M=Q_0,\dots,Q_l=M'$ and $M'=Q_0',\dots,Q_m'=N$ such that \[\sum_{i=0}^{l-1} \hat d^p_\mcI(Q_i,Q_{i+1})<\delta \quad \text{and} \quad \sum_{i=0}^{m-1} \hat d^p_\mcI(Q_i',Q_{i+1}')<\delta'.\] Concatenating the sums, we get that $d_\mcI^p(M,N)\leq \delta+\delta'$. It follows that \[d_\mcI^p(M,N)\leq d_\mcI^p(M,M')+d_\mcI^p(M',N),\] so $d_\mcI^p$ satisfies the triangle inequality, proving (i).

For any modules $M$ and $N$ and $\delta>d_\mcI^p(M,N)$, there exist modules $M=Q_0,Q_1,\dots,Q_l=N$ such that \[\sum_{i=0}^{l-1} \hat d^p_\mcI(Q_i,Q_{i+1})<\delta.\]  Let $d$ be a distance bounded above by $\hat d_\mcI^p$. Then 
\[d(M,N)\leq \sum_{i=0}^{l-1} d(Q_i,Q_{i+1}) \leq \sum_{i=0}^{l-1} \hat d^p_\mcI(Q_i,Q_{i+1}) < \delta.\] Since this holds for all $\delta>d_\mcI^p(M,N)$, we have $d(M,N)\leq d_\mcI^p(M,N)$, which proves (ii).
\end{proof}

\begin{remark}
We note that \cref{Prop:Largest_Lower_Bound} holds in more generality than stated:  Given any $\hat d:X\times X\to [0,\infty]$ satisfying all the properties of an extended pseudometric except the triangle inequality, we can construct an extended pseudometric $d:X\times X\to [0,\infty]$ from $\hat d$ exactly as in \cref{Def:Pres_Distance}.  \cref{Prop:Largest_Lower_Bound} then generalizes to $\hat d$ and $d$ with the same proof.
\end{remark}

\begin{proposition}
\label{Prop:Presentation_pq_inequality}
For finitely presented $\R^n$-persistence modules $M$, $N$ and $p\leq q\in [1,\infty]$, we have that $d^p_\mcI(M,N)\geq d^q_\mcI(M,N)$.
\end{proposition}

\begin{proof}
It is a standard fact that for any $v\in \R^n$, $\|v\|_p\geq \|v\|_q$.  The result follows from this.
\end{proof}

\begin{proposition}
$d^p_\mcI$ descends to an extended metric on isomorphism classes of finitely presented $\R^n$-persistence modules.
\end{proposition}

\begin{proof}
Clearly, $d^p_\mcI$ descends to an extended pseudometric on isomorphism classes of finitely presented $\R^n$-persistence modules.  To show that it in fact descends to an extended metric, we need to check that $d^p_\mcI(M,N)=0$ only if $M\cong N$.  It was shown in \cite[Theorem 6.1]{lesnick2015theory} that for finitely presented persistence modules $M$ and $N$, $d_{\mcI}(M,N)=0$ only if $M\cong N$.

Suppose $d^p_\mcI(M,N)=0$. We have
\[d^p_\mcI(M,N)\geq d_{\mcI}^\infty(M,N)=d_{\mcI}(M,N),\]
where the inequality follows from \cref{Prop:Presentation_pq_inequality} and the equality follows from \cref{Thm:Main_Thm_Intro}\,(i), which we prove in the next section. Thus, $d_{\mcI}(M,N)=0$, and then \cite[Theorem 6.1]{lesnick2015theory} implies that $M\cong N$.
\end{proof}

\subsection{Equality of $d_{\mcI}$ and $d^\infty_\mcI$.}
\label{Sec:int=inf}

The following result is implicit in \cite[Section 4]{lesnick2015theory}:
\begin{theorem}[Characterization of the Interleaving Distance using  Presentations]\label{Prop:Presentation_Interleaving_Equality}
For all finitely presented $n$-parameter persistence modules $M$ and $N$, $\hat d_{\mcI}^\infty(M,N)=d_{\mcI}(M,N)$.
\end{theorem}

\begin{remark}
In fact, a generalization of \cref{Prop:Presentation_Interleaving_Equality}, which holds for arbitrary (not necessarily finitely presented) $n$-parameter persistence modules, is implicit in  \cite[Section 4]{lesnick2015theory}.
\end{remark}

\cref{Thm:Main_Thm_Intro}\,(i) follows easily from \cref{Prop:Presentation_Interleaving_Equality}.  We recall the statement:
\begin{theorem1i}
For all finitely presented $n$-parameter persistence modules $M$ and $N$, $d_{\mcI}^\infty(M,N)=d_{\mcI}(M,N)$.
\end{theorem1i}

\begin{proof}
Since $d_{\mcI}$ satisfies the triangle inequality and $\hat d^\infty_\mcI=d_{\mcI}$, we see that $\hat d^\infty_\mcI$ also satisfies the triangle inequality.  \cref{Prop:Largest_Lower_Bound}\,(ii) then implies that $d^\infty_\mcI=\hat d^\infty_\mcI=d_{\mcI}$. 
\end{proof}

As noted in the introduction, the proof of \cref{Prop:Presentation_Interleaving_Equality} appearing in \cite{lesnick2015theory}, while not difficult, is rather technical and unintuitive.  Here, we present a more intuitive proof.   The key idea behind the proof is \cref{Prop:Pushes} below, which will also be of use to us later, in the proof of \cref{Thm:Main_Thm_Intro}\,(iii). 

\begin{definition}\label{Def:Push_Map}
For ${\mathcal J}\colon X\hookrightarrow Y$ an injection of posets, we say $Y$ \emph{pushes onto $\im \mathcal J$} if for any $q\in Y$, there is a unique minimal element $p\in X$ such that $q\leq \mathcal J(p)$.  If $Y$ pushes onto $\im \mathcal J$, then this defines a function \[\push^{\mathcal J}\colon Y\to X\] sending each $q\in Y$ to such $p$.
\end{definition}

In the special case that $\mathcal J$ is the inclusion of an affine line into $\R^2$, the maps $\push^{\mathcal J}$ were introduced in \cite{lesnick2015interactive} and subsequently also used in \cite{kerber2019exact,vipond2020local}.

\begin{example}\label{Ex:Interleaving_Push}
For $\delta>0$, we may regard the $\delta$-interleaving category $I^\delta$ of \cref{Sec:Interleavings} as a poset and ${\mathcal J}^0,{\mathcal J}^1\colon\R^n\to I^\delta$ as injections of posets.  Then $I^\delta$ pushes onto both $\im({\mathcal J}^0)$ and $\im({\mathcal J}^1)$.
\end{example}

\begin{proposition}\label{Prop:Pushes}
Let ${\mathcal J}\colon X\to Y$ be an injection of posets such that $Y$ pushes onto $\im \mathcal J$. A presentation matrix $P_M$ of a $Y$-persistence module $M$ induces a presentation matrix $P_M^{\mathcal J}$ of $M\circ {\mathcal J}$ with the same underlying matrix and \[\labels(P_M^{\mathcal J})=\push^{\mathcal J} \circ \labels(P_M).\]
\end{proposition}

The proof of \cref{Prop:Pushes} is straightforward; we leave the details to the reader. 

\begin{remark}
 In the case that $\mathcal J$ is the inclusion of an affine line into $\R^2$, \cref{Prop:Pushes} appears in \cite[Remark 4]{kerber2019exact}.  Recent work of Ezra Miller \cite[Proposition 6.3]{miller2020homological} gives a generalization of \cref{Prop:Pushes}, where the (co)domains of the presentations are taken to be direct sums of the indicator modules of upsets.
\end{remark}

\begin{lemma}\label{Lem:Induced_Presentations_On_Spines}
For $\delta>0$, suppose that $Z\colon I^\delta\to \Vect$ is a finitely presented $\delta$-interleaving between $\R^n$-persistence modules 
$M$ and $N$, and $P_Z$ is a presentation matrix for $Z$.  Then there exist presentation matrices $P_M$ and $P_N$ for $M$ and $N$ with the same underlying matrix as $P_Z$, such that $d^\infty(P_M,P_N)\leq \delta$.
\end{lemma}

\begin{proof}
By \cref{Ex:Interleaving_Push} and \cref{Prop:Pushes}, $P_Z$ induces a presentation of $P_M$ of $Z\circ {\mathcal J}^0=M$ and a presentation $P_N$ of $Z\circ {\mathcal J}^1=N$, both of which have the same underlying matrix as $P_Z$.  For $i\in \{0,1\}$ and $a\in \R^n$, we have 
\[\push^{{\mathcal J}^i}(a,i)=a\qquad\textup{and}\qquad\push^{{\mathcal J}^{1-i}}(a,i)=a+\delta.\]  Hence, \[\|\push^{{\mathcal J}^0}(a,i)-\push^{{\mathcal J}^1}(a,i)\|_{\infty}\leq \delta.\]  It follows that $d^\infty(P_M,P_N)\leq \delta$.  
\end{proof}

\begin{lemma}\label{Lem:Fin_Gen_Ker}
If $\gamma\colon M\to N$ is a morphism of finitely presented $\R^n$-persistence modules, then $\ker \gamma$ is finitely presented.
\end{lemma}

\begin{proof}
The corresponding result for finitely generated $\Z^n$-persistence modules is standard: $\Z^n$-persistence modules are (up to isomorphism of categories) $n$-graded modules over the polynomial ring $k[x_1,\ldots,x_n]$, and the result holds because this ring is Noetherian \cite{eisenbud1995commutative}.  

To obtain the result for $\R^n$-persistence modules, note that by \cref{Lem:Fin_Gen_LKE}, $\gamma$ is the left Kan extension of a morphism $\gamma'$ of finitely generated $\Z^n$-persistence modules along a grid function $\G\colon \Z^n \to \R^n$.  By the corresponding result for the $\Z^n$-indexed case, $\ker \gamma'$ is finitely presented.  \cref{Lem:Exact_Extension} and \cref{Prop:Free_Modules_and_Extension} together imply that a finite presentation of $\ker \gamma'$ induces a finite presentation of $\Lan_{\G}(\ker \gamma')$.  But by \cref{Lem:Exact_Extension}, we have that $ \Lan_{\G}(\ker \gamma')\cong \ker(\gamma)$, so $\ker(\gamma)$ is finitely presented. 
\end{proof}

\begin{lemma}\label{Finitely_Presented_Interleaving}
For all $\delta> 0$, a $\delta$-interleaving $Z\colon\I^\delta\to \Vect$ between finitely presented $\R^n$-persistence modules $M$ and $N$ is finitely presented.
\end{lemma}

\begin{proof}[Proof]
Since $M$ and $N$ are finitely generated, $Z$ is also finitely generated, i.e., there exists an epimorphism $\gamma\colon F\to Z$, where $F$ is a finitely generated, free $\I^\delta$-persistence module.  For $i\in \{0,1\}$,  $F\circ \mathcal J^i$
 is free and finitely generated and thus finitely presented. Additionally, $Z\circ {\mathcal J}^i$ is isomorphic to either $M$ or $N$, so is also finitely presented.  \cref{Lem:Fin_Gen_Ker} thus implies that $\ker(\gamma\circ {\mathcal J}^i)$ is finitely presented, and hence finitely generated.  But \[\ker(\gamma\circ {\mathcal J}^i)=\ker(\gamma)\circ {\mathcal J}^i,\] so both $\ker(\gamma)\circ {\mathcal J}^0$ and $\ker(\gamma)\circ {\mathcal J}^1$ are finitely generated.  This implies that $\ker \gamma$ is finitely generated, giving the result.
\end{proof}

\begin{proof}[Proof of \cref{Prop:Presentation_Interleaving_Equality}]
To show that $\hat d^\infty_\mcI(M,N)=d_{\mcI}(M,N)$, it suffices to show that for any $\delta>0$, there exists a $\delta$-interleaving between $M$ and $N$ if and only if there exist finite presentations $P_M$ and $P_N$ of $M$ and $N$ with the same underlying matrix such that $d^\infty(P_M,P_N)\leq\delta$.  The proof that if we have such presentations then we get a $\delta$-interleaving is entirely straightforward; we omit the details.  To prove the converse, note that since $M$ and $N$ are finitely presented, \cref{Finitely_Presented_Interleaving} implies that any $\delta$-interleaving $Z\colon I^\delta\to\Vect$ between $M$ and $N$ is finitely presented.  The result now follows from \cref{Lem:Induced_Presentations_On_Spines}.
\end{proof}

\subsection{Equality of $d_{\mcI}^p$ and $d_{\W}^p$ on 1-Parameter Persistence Modules}
\label{Sec:dI=dP}

We next prove \cref{Thm:Main_Thm_Intro}\,(iv).  Let us recall the statement:

\begin{theorem1iv}
For $p\in [1,\infty]$ and finitely presented 1-parameter persistence modules $M$ and $N$, 
\[d_{\mcI}^p(M,N)=d_{\W}^p(M,N)=d_\match^p(M,N).\]
\end{theorem1iv}

As mentioned in the introduction, the second equality is immediate from the definition of $d_\match^p$ (\cref{Def:Wasserstein_Matching_Distance} below), so we only consider the first equality. We first show that $d_{\W}^p(M,N) \geq d_{\mcI}^p(M,N)$ (which, as in the $p=\infty$ case, is straightforward), and then show that $d_{\W}^p(M,N) \leq d_{\mcI}^p(M,N)$.  We use the notation and conventions of \cref{Sec:Wasserstein}.  

\begin{proof}[Proof that $d_{\W}^p(M,N) \geq  d_{\mcI}^p(M,N)$]
Suppose $d_{\W}^p(M,N)=\delta<\infty$.  Then there exists a matching $\sigma\colon\B{}_M\to \B{}_N$ with $\cost(\sigma,p)=\delta$.  Let $m$ and $n$ denote the respective sizes of $\B{}_M$ and $\B{}_N$, let $l$ denote the number of pairs matched by $\sigma$, and let $q=m+n-l$.  Write 
\[\B{}_M = \{a^1,\dots, a^m\}\quad\textup{and}\quad \B{}_N = \{b^1, \dots, b^l, b^{m+1}, \dots, b^{q}\}\] where $\sigma(a^i) = b^i$ for $1\leq i \leq l$, so that $a^{l+1},\dots, a^m$ and $b^{m+1}, \dots, b^{q}$ are unmatched by $\sigma$.   For $a^i$ unmatched let $b^i=m(a^i)$, and for $b^i$ unmatched let $a^i=m(b^i)$.  

We construct a presentation $P_M$ of $M$ as follows: Starting with the $q\times q$ identity matrix, for each $a^i=(a^i_1,a^i_2)\in\B{}_M$, we assign the label $a^i_1$ to row $i$ and the label $a^i_2$ to column $i$.  We then remove all of the columns with label $\infty$.  
It is clear that this does indeed give a presentation matrix for $M$.

We construct a presentation $P_N$ of $N$ in the same way, using the $b^i$ in place of the $a^i$. Note that the deleted columns are the same as for $M$, because the assumption that $\cost(\sigma,p)=\delta<\infty$ implies that $\sigma$ matches each infinite length interval to an infinite length interval.

For $p\in [1,\infty)$, we then have
\begin{align*}
\delta&=\cost(\sigma,p)\\
&=\left(\sum_{i=1}^l \|a^i-b^i\|_p^p+\sum_{i=l+1}^m \|a^i-m(a^i)\|_p^p+\sum_{i=m+1}^q \|m(b^i)-b^i\|_p^p\right)^{\frac{1}{p}}\\
&=\left(\sum_{i=1}^q \|a^i-b^i\|_p^p\right)^{\frac{1}{p}}\\
&=\left(\sum_{i=1}^q |a_1^i-b_1^i|^p+|a_2^i-b_2^i|^p \right)^{\frac{1}{p}}\\
&= \|\labels(P_M)-\labels(P_N)\|_p\\
&=d^p(P_M,P_N)\\
&\geq \hat d_{\mcI}^p(M,N)\\
&\geq d_{\mcI}^p(M,N).
\end{align*}
In the case that $p=\infty$, a similar sequence of equations shows that $\delta\geq d_{\mcI}^\infty(M,N).$
\end{proof}

We prepare for our proof that $d_{\W}^p(M,N) \leq d_{\mcI}^p(M,N)$ by recalling some standard facts about presentations of $\R$-persistence modules.  
Let $P$ be a presentation matrix for an $\R$-persistence module $Q$. Let $P_{i*}$ and $P_{*i}$ denote the $i^{\mathrm{th}}$ row and column of $P$, respectively, and let $\la(P)_{i*}$, $\la(P)_{*i}$ denote their labels.

If $\la(P)_{*i}\leq \la(P)_{*j}$, then adding a scalar multiple of $P_{*i}$ to $P_{*j}$ yields another presentation matrix for $Q$ (with the same labels); indeed, in close analogy with ordinary linear algebra, such a column addition can be seen as representing a change of basis operation.  Similarly, if $\la(P)_{i*}\leq \la(P)_{j*}$, then adding a scalar multiple of $P_{j*}$ to $P_{i*}$ yields another presentation matrix for $Q$.  We call such row and column operations \emph{admissible operations}.

\begin{proposition}[\cite{zomorodian2005computing}]\label{Prop:SNF}
Given a presentation matrix $P$ for an $\R$-persistence module $Q$,
\begin{itemize}
\item[(i)] There exists a sequence of admissible row and column operations which transforms $P$ into  a presentation matrix $R$ with at most one non-zero entry in each row and each column.
\item[(ii)]
$\B{}_Q$ can be immediately read off $R$ via the following formula:
\begin{equation}
\begin{aligned}\label{Eq:Barcode_From_Reduced_Pres}
\B{}_Q&=\{[\la(R)_{i*},\la(R)_{*j})\mid R_{i,j}\ne 0\textup{ and }\la(R)_{i*}\ne \la(R)_{*j}\} \\
       & \cup \{[\la(R)_{i*},\infty)\mid R_{i*}=0\}.
\end{aligned}      
\end{equation}
\end{itemize}

\end{proposition}

One says that the presentation $R$ of \cref{Prop:SNF}\,(i) is in \emph{(graded Smith) normal form}.

\begin{definition}\label{Def:Order_Refines}
Given functions $f,g\colon S\to \R$, we will say that $g$ is $f$\emph{-compatible} if $f(x)\leq f(y)$ implies $g(x)\leq g(y)$ for all $x,y\in S$.
\end{definition}

Note that if $g$ is $f$-compatible, then $f(x)<f(y)$ implies $g(x)\leq g(y)$, and $f(x)=f(y)$ implies $g(x)=g(y)$.  

The following lemma is a straightforward consequence of the definitions.
\begin{lemma}\label{Lem:Refinements_And_Admissibility}
If $P$ and $P'$ are presentation matrices for $\R$-persistence modules such that $\labels(P')$ is $\labels(P)$-compatible, then a sequence of admissible row and column operations on $P$ is also admissible on $P'$.
\end{lemma}

We are now ready to finish the proof of \cref{Thm:Main_Thm_Intro} (iv).
\begin{proof}[Proof that $d_{\W}^p(M,N) \leq  d_{\mcI}^p(M,N)$]

According to \cref{Prop:Largest_Lower_Bound}, $d_{\mcI}^p$ is the largest distance less than or equal to $\hat{d}_\mcI^p$.  Thus, since $d_{\W}^p$ is a distance, it is enough to show that $d_{\W}^p(M,N)\leq \hat d_{\mcI}^p(M,N).$

Suppose that $P_M$ and $P_N$ are finite presentations of modules $M$ and $N$ with the same underlying $r\times c$ matrix. For $t\in [0,1]$, let $P_t$ be the presentation with the same underlying matrix as $P_M$ and $P_N$, and \[\labels(P_t) = (1-t)\labels(P_M) + t\cdot\labels(P_N).\] Note that $P_0=P_M$ and $P_1=P_N$.  Let $M_t$ be a module $P_t$ is presenting.  There exists a finite set of real numbers \[0=t_0<t_1<t_2<\dots<t_{w+1}=1\]
such that for $i\in \{0,\ldots, w\}$ and $s\in (t_i,t_{i+1})$, both $\labels(P_{t_i})$ and $\labels(P_{t_{i+1}})$ are $\labels(P_s)$-compatible.  Explicitly, we may take $\{t_1,\ldots,t_w\}$ to be the set of points $t\in (0,1)$ such that there exist $i,j\in \{1,\ldots,r+c\}$ and $t'\in [0,1]$ with
\[\labels(P_{t})_i=\labels(P_{t})_j\quad\textup{and}\quad\labels(P_{t'})_i\ne \labels(P_{t'})_j.\]  
Informally, this is the set of points where the order of the labels changes as $t$ increases.

By \cref{Prop:SNF}\,(i), there exists a sequence of admissible row and column operations putting $P_s$ into Smith normal form, and by \cref{Lem:Refinements_And_Admissibility} this sequence is also admissible for both $P_{t_i}$ and $P_{t_{i+1}}$.  In particular, this gives normal form presentations $R_0$ and $R_{1}$ for $M_{t_i}$ and $M_{t_{i+1}}$, respectively, with the same underlying matrix $R$ and 
\begin{equation}\label{Eq:Labels_are_Equal}
\labels(R_0)=\labels(P_{t_i}),\qquad \labels(R_{1})=\labels(P_{t_{i+1}}).
\end{equation}
By \cref{Prop:SNF}\,(ii), this pair of presentations induces a matching $\sigma_i\colon\B{}_{M_{t_i}}\to \B{}_{M_{t_{i+1}}}$.

We claim that \[\cost(\sigma_i,p) \leq \|\labels(R_0)-\labels(R_{1})\|_p.\]
To prove the claim, let $z\in \{0,1\}$.  
Note that by \cref{Prop:SNF}\,(ii), each $a=(a_1,a_2)\in \B{}_{M_{t_{i+z}}}$ with $a_2<\infty$ is indexed by a row-column pair $(r_a,c_a)$ of $R$ such that
\[a_1=\labels(R_z)_{r_a}\quad\textup{and}\quad a_2=\labels(R_{z})_{c_a},\]
where have simplified notation by letting $\labels(R_z)_{r_a}=\labels(R_z)_{r_a*}$ and $\labels(R_z)_{c_a}=\labels(R_z)_{*c_a}.$ 
Similarly, each $a=(a_1,\infty)\in \B{}_{M_{t_{i+z}}}$ is indexed by a zero row $r_a$ of $R$ such that \[a_1=\labels(R_z)_{r_a}.\]

It follows that for each matched pair $(a,b)\in \sigma_i$ with $a_2,b_2<\infty$,  
\begin{align*}
\|a-b\|_p^p&=|a_1-b_1|^p+|a_2-b_2|^p\\
&=|\labels(R_0)_{r_a}-\labels(R_1)_{r_a}|^p+|\labels(R_0)_{c_a}-\labels(R_1)_{c_a}|^p,
\end{align*}
and for each $(a,b)\in \sigma_i$ with $a_2=\infty=b_2$,  
\[\|a-b\|_p^p=|a_1-b_1|^p=|\labels(R_0)_{r_a}-\labels(R_1)_{r_a}|^p.\]
If $a\in \B{}_{M_{t_z}}$ is not matched by $\sigma_i$, then letting 
\[m'=\labels(R_{1-z})_{r_a}=\labels(R_{1-z})_{c_a},\] 
an easy calculation gives that
\begin{align*}
\|a-m(a)\|^p_p&\leq \|a-(m',m')\|_p^p\\
&=|\labels(R_0)_{r_a}-\labels(R_1)_{r_a}|^p+|\labels(R_0)_{c_a}-\labels(R_1)_{c_a}|^p.
\end{align*}

Now observe that for all $a,a'\in B{}_{M_{t_{i+z}}}$, we have $r_a=r_{a'}$ only if $a=a'$.  Similarly, $c_a=c_{a'}$ only if $a=a'$.  This implies the second equality below:
\begin{align*}
\cost(\sigma_i,p)^p &= \sum_{(a,b)\in \sigma_i} \|a-b\|_p^p+\sum_{a \textup{ unmatched by $\sigma_i$}} \|a-m(a)\|^p\\
&\leq  \sum_{(a,b)\in \sigma_i} |\labels(R_0)_{r_a}-\labels(R_1)_{r_a}|^p+ \sum_{\substack{(a,b)\in \sigma_i,\\ a_2,b_2<\infty}}  |\labels(R_0)_{c_a}-\labels(R_1)_{c_a}|^p\\
&+\sum_{a \textup{ unmatched by $\sigma_i$}} |\labels(R_0)_{r_a}-\labels(R_1)_{r_a}|^p+|\labels(R_0)_{c_a}-\labels(R_1)_{c_a}|^p\\
&\leq \sum_{j=1}^{r+c} |\labels(R_0)_j-\labels(R_1)_j|^p\\
&= \|\labels(R_0)-\labels(R_1)\|_p^p.
\end{align*}
The claim now follows by taking $p^{\mathrm{th}}$ roots.

In view of \cref{Eq:Labels_are_Equal}, we have \[\|\labels(R_0)-\labels(R_1)\|_p= \|\labels(P_{t_i})-\labels(P_{t_{i+1}})\|_p.\]  Thus, 
\begin{align*}
d_{\W}^p(M,N) &\leq \sum_{i=0}^w \cost(\sigma_i,p)\\
&\leq \sum_{i=0}^w \|\labels(P_{t_i})-\labels(P_{t_{i+1}})\|_p\\
&= \|\labels(P_M)-\labels(P_N)\|_p\\
&=d^p(P_M,P_N),\\
\end{align*}
where the second to last equality holds because the presentations $P_{t_i}$ were constructed via linear interpolation on the labels.  Since $(P_M,P_N)$ was an arbitrary pair in $P_{M,N}$, we are done.
\end{proof}

\section{The $p$-Matching Distance}
\label{Sec:W_matching}
As noted in the introduction, the definition of the $p$-matching distance is given simply by replacing $d^\infty_\W$ with $d_{\W}^p$ in the definition of the matching distance:

\begin{definition}\label{Def:Wasserstein_Matching_Distance}
For $p\in [0,\infty]$, the $p$-matching distance between \pfd $\R^n$-persistence modules $M$ and $N$ is given by 
\begin{align*}
d^p_\match(M,N) = \sup_{l}\, d_{\W}^p(M^l,N^l),
\end{align*}
where $l\colon \R\to \R^n$ ranges over all admissible lines (as defined in \cref{Sec:Matching_Distance}).
\end{definition}
Using the triangle inequality for $d_{\W}^p$, it is easy to check that $d_{\match}^p$ satisfies the triangle inequality, and hence is indeed a distance (i.e., an extended pseudometric). 
It is clear from the definition that \[d^\infty_\match(M,N)=d_\match(M,N).\]

\subsection{The $p$-Matching Distance Bounds the $p$-Presentation Distance}

We now prove \cref{Thm:Main_Thm_Intro}\,(iii), which extends Landi's inequality $d_{\match}\leq d_{\mcI}$ (\cref{Prop:Matching_Distance_Lower_Bound}).  Recall the statement:
\begin{theorem1iii}
For all finitely presented $\R^n$-persistence modules $M$, $N$, \[d_\match^p(M,N)\leq d_{\mcI}^p(M,N).\]
\end{theorem1iii}

Note that if $l:\R\to \R^n$ is an admissible line, then $\R^n$ pushes onto $\im l$, in the sense of \cref{Def:Push_Map}.  

\begin{lemma}\label{Lem:Push_Stability}
For any $a,b\in \R^n$, admissible line $l$, and $p\in [1,\infty]$, we have 
\[|\push^l(a)-\push^l(b)|\leq \|a-b\|_p.\] 
\end{lemma}

\begin{proof}
Since $\|a-b\|_\infty<\|a-b\|_p$, it suffices to prove the result in the case $p=\infty$.  Let $l$ be given by $l(t)=t\vec v+\vec w$, and let $\delta=\|a-b\|_\infty$.  Recalling our notation 
$\vec \delta=(\delta,\delta,\ldots,\delta)\in \R^n$ from \cref{Sec:Interleavings}, 
note that $a-b \leq \vec \delta$. Since $\vec 1\leq \vec v$, we have
\begin{align*}
a&=b+(a-b)\\
&\leq b+\vec\delta\\
&\leq  l\circ \push^l(b)+\vec\delta\\
&\leq l\circ \push^l(b)+\delta\vec v.
\end{align*}
Thus, $l\circ \push^l(b)+\delta\vec v$ is a point on $\im l$ that is greater than or equal to $a$, so
\[l\circ \push^l(a)\leq l\circ \push^l(b)+\delta \vec v.\]  The same argument with $a$ and $b$ switched shows that
\[l\circ \push^l(b)\leq l\circ \push^l(a)+\delta \vec v.\] It then follows from the definitions of $l$ that
\begin{align*}
\push^l(a)&\leq \push^l(b)+\delta,\\
\push^l(b)&\leq \push^l(a)+\delta.
\end{align*}
This gives $|\push^l(a)-\push^l(b)|\leq \delta$, as desired.
\end{proof}

\begin{proof}[Proof of \cref{Thm:Main_Thm_Intro}\,(iii)]
Recall from \cref{Prop:Largest_Lower_Bound} that $d_\mcI^p$ is the largest distance bounded above by $\hat d_\mcI^p$. Since $d_\match^p$ is also a distance, it suffices to show that  $d_\match^p(M,N) \leq \hat d_\mcI^p(M,N)$, i.e., that for any admissible line $l$,
\begin{equation}\label{Eq:Lower_Bound}
d_{\W}^p(M^l,N^l) \leq \hat d_{\mcI}^p(M,N).
\end{equation}

By \cref{Prop:Pushes}, any pair of presentation matrices $P_M$ and $P_N$ for $M$ and $N$ induces presentation matrices $P_{M^l},P_{N^l}$ for $M^l$ and $N^l$, with 
\[\labels(P_{M^l})=\push^l\circ \labels(P_{M})\quad\textup{and}\quad\labels(P_{N^l})=\push^l\circ \labels(P_{N}).\]
\cref{Lem:Push_Stability} implies that $d^p(P_{M^l},P_{N^l})\leq d^p(P_M,P_N)$.  Hence, $\hat d_{\mcI}^p(M^l,N^l)\leq \hat d_{\mcI}^p(M,N)$.  Since $M^l$ and $N^l$ are $1$-parameter persistence modules, \cref{Thm:Main_Thm_Intro}\,(iv) gives us that \[d_{\W}^p(M^l,N^l)=d_{\mcI}^p(M^l,N^l)=\hat d_{\mcI}^p(M^l,N^l),\] completing the proof.
\end{proof}

\subsection{Computing the $p$-Matching Distance on Bipersistence Modules}\label{Sec:Computation}
Assume we are given presentations $P_M$ and $P_N$ for $\R^2$-persistence modules $M$ and $N$, where the total number of generators and relations in both presentations is $m$.  By translating  both modules if necessary, we can assume without loss of generality that all grades of generators and relations for both modules lie in some bounded rectangle $[0,X]\times [0,Y]$.  Let $C=\max\,\{X,Y\}$.  

An algorithm of Biasotti et al. \cite{biasotti2011new} computes an estimate $D$ of $d_{\match}(M,N)$ satisfying \[d_{\match}(M,N)-\epsilon\leq D \leq d_{\match}(M,N)\] in $O\left(\frac{(mC)^3}{\epsilon^2}\right)$ time.  In brief, the idea is to take \[D=\max_{l\in L}\ d^\infty_{\W}(M^l,N^l),\] where $L$ is a suitably chosen finite set of admissible lines; each of the bottleneck distances $d^\infty_{\W}(M^l,N^l)$ can be computed efficiently using a standard algorithm, e.g., the one of \cite{kerber2017geometry}.  
$L$ is in fact chosen adaptively, using a quad tree construction.  The guarantee that $d_{\match}(M,N)-\epsilon\leq D$ follows from the algebraic stability theorem.  

More recent work of Kerber and Nigmetov \cite{kerber2020efficient} revisits these ideas, introducing several optimizations to the approach of  \cite{biasotti2011new}.  This leads  to a more efficient algorithm, with time complexity $O\left(m^3\left(\frac{C}{\epsilon}\right)^2\right)$.  It turns out that, using \cref{Thm:Main_Thm_Intro}\,(iv), the algorithm of \cite{kerber2020efficient} extends quite readily to an approximation algorithm for $d^p_{\match}(M,N)$, for any $p\in [1,\infty]$, yielding the following result:

\begin{proposition}[Efficient Approximation of the $p$-Matching Distance]\label{Prop:MAtching_Distance_Approx}
For any $p\in [1,\infty]$ and $\epsilon>0$, an approximation $D$ of $d^p_{\match}(M,N)$, satisfying
\[d^p_{\match}(M,N)-\epsilon\leq D \leq d^p_{\match}(M,N),\] can be computed in time \[O\left(m^{3+\frac{2}{p}}\left(\frac{C}{\epsilon}\right)^2\right).\]
\end{proposition}

\begin{proof}
We briefly outline the modifications to the approach of \cite{kerber2020efficient} which lead to this result.  First, let us remark that while the exposition of \cite{kerber2020efficient} assumes that the input to the algorithm is a pair of bifiltrations, the authors note in the conclusion that their algorithm in fact applies to presentations in exactly the same way.  

In \cite{kerber2020efficient}, Kerber and Nigmetov in fact introduce and compare a few variants of their algorithm; for simplicity, we focus here on the variant which uses the so-called \emph{local linear bound.}
In this variant, at each rectangle $B$ constructed in the quad-tree, quantities $v(M,B)$ and $v(N,B)$ are computed; we now define these quantities.  Elements of $B$ parameterize admissible lines; to keep notation simple, let us in fact identify each element of $B$ with the line it parameterizes, and let $l_c\in B$ denote the center of $B$.  Then, adapting the notation of  \cite{kerber2020efficient} to the notation for presentations used in our paper, the definition is as follows:
\[v(M,B)=\max_{a\in \labels(P_M)} v(a,B),\]
where 
\[v(a,B)=\max_{l\in B} |\push^l(a)-\push^{l_c}(a)|.\] 
We define $v(N,B)$ analogously. 

To extend the algorithm of \cite{kerber2020efficient}, just two simple changes are required: First, wherever a bottleneck distance is computed in the original algorithm, we instead compute a $p$-Wasserstein distance.  Second, in the specification of the algorithm, we replace $v(M,B)$ and $v(N,B)$ with quantities $v_p(M,B)$ and $v_p(N,B)$, defined in the following way: Regarding $\labels(P_M)$ as a vector $(a_1,\ldots,a_z)\in (\R^2)^z$, we let \[v_p(M,B)=\|v(a_1,B),v(a_2,B),\ldots,v(a_z,B))\|_p.\]  $v_p(N,B)$ is defined in the same way.  Note that $v_\infty(M,B)=v(M,B)$, so in the $p=\infty$ case, this modified algorithm is indeed the same as the one in \cite{kerber2020efficient}. 

The proof of correctness of the algorithm given in \cite{kerber2020efficient} extends to the case of arbitrary $p$ with only one substantive change: Where  \cite{kerber2020efficient} (implicitly) invokes the algebraic stability theorem, we need to instead apply the more general bound $d_\W^p\leq d^p_{\mcI}$ implied by \cref{Thm:Main_Thm_Intro}\,(iv).  
Using the fact that \[v_p(M,B)\leq z^\frac{1}{p}v(M,B)\leq m^\frac{1}{p}v(M,B),\] the complexity analysis of \cite{kerber2020efficient} also extends readily to give our claimed runtime bound.
\end{proof}

\section{Universality of the Presentation Distance}
\label{Sec:Universality}
In this section, we prove \cref{Thm:Universality}, our stability and universality result for the presentation distance on 1- and 2-parameter persistence modules.  We also show that the constant  $n^\frac{1}{p}$ appearing in the statement of \cref{Thm:Universality}\,(i) is tight.  

Let us recall the statement of the theorem, keeping in mind our convention $\frac{1}{\infty} = 0$:

\begin{reptheorem}{Thm:Universality}
For any $p\in [1,\infty]$ and $n\in \{1,2\}$,
\begin{itemize}
\item[(i)] $d_{\mcI}^p$ is $p$-stable, and also $p$-stable across degrees with constant $n^\frac{1}{p}$, on finitely presented $n$-parameter persistence modules,
\item[(ii)] if the field of coefficients $k$ is prime and $d$ is any $p$-stable distance on finitely presented $n$-parameter persistence modules, then $d\leq d_{\mcI}^p$.
\end{itemize}
\end{reptheorem}

Most of our effort will go into proving \cref{Thm:Universality}\,(i).  

\subsection{Proof of Stability}\label{Sec:Proof_of_Stability}
Our proof of \cref{Thm:Universality}\,(i) will focus primarily on the $n=2$ case, as a slight variant of our argument for this case also handles the $n=1$ case; we discuss this variant at the end of the proof.  As noted in the introduction, the $n=1$ case also follows immediately from Skraba and Turner's result \cref{Thm:Skraba_Turner_Stability} and the inequality $d_{\mcI}^p \leq d^p_{\W}$ of \cref{Thm:Main_Thm_Intro}\,(iv).  

We will need the following notation: 

\begin{definition}[Chain Complex of a Monotone Function]\label{Def:Chain_Complex}
For $X$ a finite CW-complex, let $X^j$ denote the set of $j$-cells of $X$.  For any $n\geq 0$, a monotone function $f\colon\Cells(X)\to \R^n$ has an associated chain complex of free $n$-parameter persistence modules 
\[\C(f)=\quad \cdots \xrightarrow{\partial^f_{j+1}} \C_j(f)\xrightarrow{\partial^f_{j}} \C_{j-1} (f)\xrightarrow{\partial^f_{j-1}} \cdots \xrightarrow{\partial^f_2} \C_1(f) \xrightarrow{\partial^f_{1}} \C_0(f),\]
where 
\[\C_j(f)=\bigoplus_{\sigma\in {X^j}} Q^{f(\sigma)},\]
and each boundary map $\delta^f_j$ is induced by the $j^{\mathrm{th}}$ cellular boundary map of $X$.  
\end{definition}
We note that for all $j\geq 0$, we have \[H_j(\C(f))=H_j\circ \mathcal S(f);\]
that is, the $j^{\mathrm{th}}$ homology module of the chain complex is equal to the composition of the sublevel filtration with the $j^{\mathrm{th}}$ cellular homology functor. 

We begin with a brief outline of the proof of  \cref{Thm:Universality}\,(i).  Given functions $f,g:\Cells(X)\to \R^2$, each module $\ker \partial^f_j$ and $\ker \partial^g_j$ is free by a result in \cite{chacholski2017combinatorial}.  Using an interpolation argument similar to the one used in the proofs of 
 \cref{Thm:Skraba_Turner_Stability} from \cite{skraba2020wasserstein} and \cref{Thm:Main_Thm_Intro}\,(iv), we observe that it suffices to consider the case where the sublevel filtration  $\mathcal S(g)$ is a left Kan extension of  $\mathcal S(f)$ along a grid function $\G\colon \R^2\to \R^2$.  In this case, a basis of $\ker \partial^f_j$ induces a corresponding  basis of $\ker \partial^g_j$.  This  gives us presentations for $H_j\mathcal S(f)$ and $H_j\mathcal S(g)$ with the same underlying matrix.  Using a Gr\"obner basis argument, we show that for each $j$-cell $\sigma$ of $X$, changing the $x$-coordinate of $f(\sigma)$ to that of $g(\sigma)$ induces a corresponding change in the $x$-label of at most one row in the presentation.  Moreover, distinct $j$-cells induce changes to distinct labels.  The analogous statements with $y$ replacing $x$ are also true.  Given this, the result follows readily.
 
Before proceeding with the details of the proof, we give an example which illustrates the result and illuminates parts of the proof.  We will also use the same example in \cref{Prop:Tightness} to show that the constant $n^{\frac{1}{p}}$ appearing in the statement of the proof is tight.  

\begin{example}\label{Ex:Stability}
Let $X$ be a cell complex with two $0$-cells and three $1$-cells, and assume that each $1$-cell is attached to both $0$-cells, as illustrated in \cref{Fig:stability_filtration}\,(A).  Consider a function $f\colon\Cells(X)\to \R^2$ where the $1$-cells map to $(1,4)$, $(3,3)$ and $(4,1)$, and both $0$-cells map to $(0,0)$.  Let $\sigma$ denote the 1-cell mapping to $(3,3)$, and let $g\colon\Cells(X)\to \R^2$ be the function which is identical to $f$ except that $g(\sigma)=(2,2)$.  Both $f$ and $g$ are monotone.  The two functions are shown in \cref{Fig:stability_filtration}\,(A).

\cref{Fig:stability_filtration}\,(B) depicts $H_1\mathcal{S}(f)$ and $H_0\mathcal{S}(f)$.  These are the kernel and cokernel, respectively, of the boundary map $\partial_1^f\colon \C_1(f) \to \C_0(f)$. We see that 
$H_1\mathcal{S}(f)$ is a free module on two generators 
at $(3,4)$ and $(4,3)$, and that $H_0\mathcal{S}(f)$ is a direct sum of two indecomposables, one a free module on a single generator at $(0,0)$, and the other a module with a generator at $(0,0)$ and relations at $(1,4)$, $(3,3)$ and $(4,1)$.  Thus, the following is a presentation of $H_0\mathcal{S}(f)$:
\[
P^f=\begin{blockarray}{cccc}
  & (1,4)& (3,3) & (4,1)  \\
\begin{block}{c[ccc]}
(0,0) &0 &  0 & 0  \\
(0,0) &1 &  1 & 1  \\
   \end{block}
   \end{blockarray}.
   \]
Similarly, \cref{Fig:stability_filtration}\,(C) depicts $H_1\mathcal{S}(g)$ and $H_0\mathcal{S}(g)$.  We see that $H_1\mathcal{S}(g)$ is a free module on two generators 
at $(2,4)$ and $(4,2)$, and that $H_0\mathcal{S}(g)$ is a direct sum of two indecomposables, one a free module on a single generator at $(0,0)$, and the other a module with a generator at $(0,0)$ and relations at $(1,4)$, $(2,2)$ and $(4,1)$.  Thus, the following is a presentation of $H_0\mathcal{S}(g)$:
\[
P^g=\begin{blockarray}{ccccc}
  & (1,4)& (2,2) & (4,1)  \\
\begin{block}{c[cccc]}
(0,0) &0 &  0 & 0  \\
(0,0) &1 &  1 & 1  \\
   \end{block}
   \end{blockarray}.
   \]
We see that for $p\in [1,\infty]$, $d^p(P^f,P^g)=(1^p+1^p)^{\frac{1}{p}}=2^{\frac{1}{p}}$.  Hence, \[d_{\mcI}^p(H_0\mathcal{S}(f),H_0\mathcal{S}(g))\leq 2^{\frac{1}{p}}=\|f-g\|_p,\] as guaranteed by \cref{Thm:Universality}\,(i).  Similarly, we see that \[d_{\mcI}^p(H_1\mathcal{S}(f),H_1\mathcal{S}(g))\leq 2^{\frac{1}{p}}=\|f-g\|_p,\]
and that \[\|d_{\mcI}^p(H_{*}\mathcal{S}(f),H_{*}\mathcal{S}(g))\|_p\leq \|(2^\frac{1}{p},2^\frac{1}{p})\|_p= 4^{\frac{1}{p}}=2^{\frac{1}{p}} \|f-g\|_p,\]
as also guaranteed by the theorem.  In the proof of \cref{Prop:Tightness}, we will show that each of the inequalities above is in fact an equality.
 
We see in this example that changing the value of $f(\sigma)$ from $(3,3)$ to $(2,2)$ changes the grades of two generators of $H_1\mathcal S(f)$:  The $x$-coordinate of one generator and $y$-coordinate of the other generator both change from 3 to 2.  In addition, the grade of one relation of $H_0\mathcal S(f)$ changes from $(3,3)$ to $(2,2)$.
\end{example}

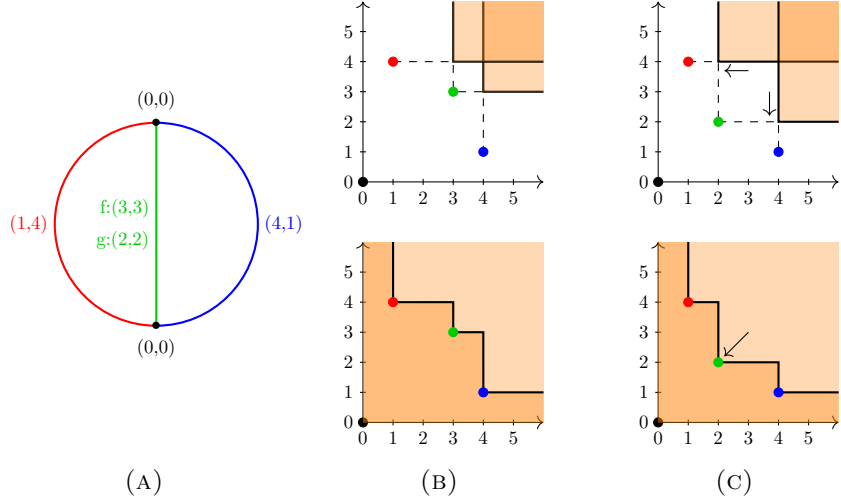
\begin{figure}[t!]
    \centering
\begin{subfigure}[t]{0.3\textwidth}
    \centering
\begin{tikzpicture}[scale=.45,every node/.style={scale=.65}]
\draw[thick,color=red] (0,0) to [out=180,in=270] (-3,3) to [out=90,in=180] (0,6);
\draw[thick,color=green!80!black] (0,0) to (0,6);
\draw[thick,color=blue] (0,0) to [out=0,in=270] (3,3) to [out=90,in=0] (0,6);
\node at (0,0){$\bullet$};
\node at (0,6){$\bullet$};
\node at (0,-3.5){};
\node at (0,12){};
\node[left,color=red] at (-3,3){(1,4)};
\node[right,color=blue] at (3,3){(4,1)};
\node[above left,color=green!80!black] at (0,3){f:(3,3)};
\node[below left,green,color=green!80!black] at (0,3){g:(2,2)};
\node[above] at (0,6.2){(0,0)};
\node[below] at (0,-.2){(0,0)};
\end{tikzpicture}
\caption{}
\end{subfigure}
\begin{subfigure}[t]{0.3\textwidth}
    \centering
\begin{tikzpicture}[scale=.4,every node/.style={scale=.65}]
\begin{scope}
\draw[<->] (0,6) to (0,0) to (6,0);
\draw[dashed] (1,4) to (3,4) to (3,3) to (4,3) to (4,1);

\draw (-.1,0) to (.1,0);
\node[left] at (-.1,0){$0$};
\draw (-.1,1) to (.1,1);
\node[left] at (-.1,1){$1$};
\draw (-.1,2) to (.1,2);
\node[left] at (-.1,2){$2$};
\draw (-.1,3) to (.1,3);
\node[left] at (-.1,3){$3$};
\draw (-.1,4) to (.1,4);
\node[left] at (-.1,4){$4$};
\draw (-.1,5) to (.1,5);
\node[left] at (-.1,5){$5$};

\draw (0,-.1) to (0,.1);
\node[below] at (0,-.1){$0$};
\draw (1,-.1) to (1,.1);
\node[below] at (1,-.1){$1$};
\draw (2,-.1) to (2,.1);
\node[below] at (2,-.1){$2$};
\draw (3,-.1) to (3,.1);
\node[below] at (3,-.1){$3$};
\draw (4,-.1) to (4,.1);
\node[below] at (4,-.1){$4$};
\draw (5,-.1) to (5,.1);
\node[below] at (5,-.1){$5$};
\draw[fill=black] (0,0) circle (.15);
\draw[thick] (6,3) to (4,3) to (4,6);
\fill[orange, opacity=0.3]
(4,3) to (4,6) to (6,6) to (6,3) to (4,3);
\draw[thick] (3,6) to (3,4) to (6,4);
\fill[orange, opacity=0.3]
(3,4) to (3,6) to (6,6) to (6,4) to (3,4);
\draw[color=red,fill=red] (1,4) circle (.15);
\draw[color=blue,fill=blue] (4,1) circle (.15);
\draw[color=green!80!black,fill=green!80!black] (3,3) circle (.15);
\end{scope}
\begin{scope}[yshift=-8cm]
\draw[<->] (0,6) to (0,0) to (6,0);
\draw[dashed] (2,4) to (3,4) to (3,3) to (4,3) to (4,2);

\draw (-.1,0) to (.1,0);
\node[left] at (-.1,0){$0$};
\draw (-.1,1) to (.1,1);
\node[left] at (-.1,1){$1$};
\draw (-.1,2) to (.1,2);
\node[left] at (-.1,2){$2$};
\draw (-.1,3) to (.1,3);
\node[left] at (-.1,3){$3$};
\draw (-.1,4) to (.1,4);
\node[left] at (-.1,4){$4$};
\draw (-.1,5) to (.1,5);
\node[left] at (-.1,5){$5$};

\draw (0,-.1) to (0,.1);
\node[below] at (0,-.1){$0$};
\draw (1,-.1) to (1,.1);
\node[below] at (1,-.1){$1$};
\draw (2,-.1) to (2,.1);
\node[below] at (2,-.1){$2$};
\draw (3,-.1) to (3,.1);
\node[below] at (3,-.1){$3$};
\draw (4,-.1) to (4,.1);
\node[below] at (4,-.1){$4$};
\draw (5,-.1) to (5,.1);
\node[below] at (5,-.1){$5$};
\draw[fill=black] (0,0) circle (.15);
\fill[orange, opacity=0.3]
(0,6) to (0,0) to (6,0) to (6,6) to (0,6);
\fill[orange, opacity=0.3]
(0,6) to (1,6) to (1,4) to (3,4) to (3,3) to (4,3) to (4,1) to (6,1) to (6,0) to (0,0) to (0,6);
\draw[thick] (1,6) to (1,4) to (3,4) to (3,3) to (4,3) to (4,1) to (6,1);
\draw[color=red,fill=red] (1,4) circle (.15);
\draw[color=blue,fill=blue] (4,1) circle (.15);
\draw[color=green!80!black,fill=green!80!black] (3,3) circle (.15);
\end{scope}
\end{tikzpicture}
\caption{}
\end{subfigure}
\begin{subfigure}[t]{0.3\textwidth}
    \centering
\begin{tikzpicture}[scale=.4,every node/.style={scale=.65}]
\begin{scope}
\draw[<->] (0,6) to (0,0) to (6,0);
\draw[dashed] (1,4) to (2,4) to (2,2);
\draw[dashed] (2,2) to (4,2) to (4,1);
\draw (-.1,0) to (.1,0);
\node[left] at (-.1,0){$0$};
\draw (-.1,1) to (.1,1);
\node[left] at (-.1,1){$1$};
\draw (-.1,2) to (.1,2);
\node[left] at (-.1,2){$2$};
\draw (-.1,3) to (.1,3);
\node[left] at (-.1,3){$3$};
\draw (-.1,4) to (.1,4);
\node[left] at (-.1,4){$4$};
\draw (-.1,5) to (.1,5);
\node[left] at (-.1,5){$5$};

\draw (0,-.1) to (0,.1);
\node[below] at (0,-.1){$0$};
\draw (1,-.1) to (1,.1);
\node[below] at (1,-.1){$1$};
\draw (2,-.1) to (2,.1);
\node[below] at (2,-.1){$2$};
\draw (3,-.1) to (3,.1);
\node[below] at (3,-.1){$3$};
\draw (4,-.1) to (4,.1);
\node[below] at (4,-.1){$4$};
\draw (5,-.1) to (5,.1);
\node[below] at (5,-.1){$5$};
\draw[fill=black] (0,0) circle (.15);
\fill[orange, opacity=0.3]
(6,4) to (2,4) to (2,6) to (6,6) to (6,4);
\draw[thick] (6,4) to (2,4) to (2,6);
\fill[orange, opacity=0.3]
(6,2) to (4,2) to (4,6) to (6,6) to (6,2);
\draw[thick] (6,2) to (4,2) to (4,6);
\draw[->] (3,3.7) to (2.2,3.7);
\draw[->] (3.7,3) to (3.7,2.2);
\draw[color=red,fill=red] (1,4) circle (.15);
\draw[color=blue,fill=blue] (4,1) circle (.15);
\draw[color=green!80!black,fill=green!80!black] (2,2) circle (.15);
\end{scope}
\begin{scope}[yshift=-8cm]
\draw[<->] (0,6) to (0,0) to (6,0);

\draw (-.1,0) to (.1,0);
\node[left] at (-.1,0){$0$};
\draw (-.1,1) to (.1,1);
\node[left] at (-.1,1){$1$};
\draw (-.1,2) to (.1,2);
\node[left] at (-.1,2){$2$};
\draw (-.1,3) to (.1,3);
\node[left] at (-.1,3){$3$};
\draw (-.1,4) to (.1,4);
\node[left] at (-.1,4){$4$};
\draw (-.1,5) to (.1,5);
\node[left] at (-.1,5){$5$};

\draw (0,-.1) to (0,.1);
\node[below] at (0,-.1){$0$};
\draw (1,-.1) to (1,.1);
\node[below] at (1,-.1){$1$};
\draw (2,-.1) to (2,.1);
\node[below] at (2,-.1){$2$};
\draw (3,-.1) to (3,.1);
\node[below] at (3,-.1){$3$};
\draw (4,-.1) to (4,.1);
\node[below] at (4,-.1){$4$};
\draw (5,-.1) to (5,.1);
\node[below] at (5,-.1){$5$};

\draw[fill=black] (0,0) circle (.15);
\fill[orange, opacity=0.3]
(0,6) to (0,0) to (6,0) to (6,6) to (0,6);
\fill[orange, opacity=0.3]
(0,6) to (1,6) to (1,4) to (2,4) to (2,2) to (4,2) to (4,1) to (6,1) to (6,0) to (0,0) to (0,6);
\draw[thick] (1,6) to (1,4) to (2,4) to (2,2) to (4,2) to (4,1) to (6,1);
\draw[->] (3,3) to (2.2,2.2);
\draw[color=red,fill=red] (1,4) circle (.15);
\draw[color=blue,fill=blue] (4,1) circle (.15);
\draw[color=green!80!black,fill=green!80!black] (2,2) circle (.15);
\end{scope}
\end{tikzpicture}
\caption{}
\end{subfigure}
\caption{The sublevel filtrations of \cref{Ex:Stability} and their non-trivial homology modules.  (A) The CW-complex $X$ and functions $f,g\colon\Cells(X)\to \R^2$.  The functions have the same value on every cell except for the green 1-cell in the center.  (B) $H_1\mathcal{S}(f)$ (above) and $H_0\mathcal{S}(f)$ (below); the light shading denotes the region where the vector spaces have dimension 1, and the dark shading denotes the region where they have dimension 2.  (C) $H_1\mathcal{S}(g)$  and $H_0\mathcal{S}(g)$.}
\label{Fig:stability_filtration}
\end{figure}

\paragraph{Kernels of Morphisms of Free Bipersistence Modules}
Throughout this section, the term \emph{bipersistence module} will refer to an $\R^2$-persistence module.  

\begin{lemma}[\cite{chacholski2017combinatorial}]\label{Lem:Free_Kernels}
If $\gamma\colon F\to G$ is a morphism of finitely generated free bipersistence modules, then $\ker \gamma$ is free.
\end{lemma}

\begin{proof}
The corresponding result for $\Z^2$-persistence modules is proven in \cite{chacholski2017combinatorial}.  To obtain the stated result, note that by \cref{Lem:Fin_Gen_LKE}, $\gamma$ is the left Kan extension of a morphism $\gamma'$ of finitely generated, free $\Z^2$-persistence modules along some grid function $\G$.  By the result for $\Z^2$-persistence modules, $\ker \gamma'$ is free.  \cref{Lem:Exact_Extension} then implies that $\ker \gamma=\Lan_{\G}(\ker \gamma')$, and so $\gamma$ is free by \cref{Prop:Free_Modules_and_Extension}\,(i).
 \end{proof}

Recall from \cref{Sec:Free_Mods} that if $B=(b_1,\dots,b_{|B|})$ is an ordered basis of a finitely generated free bipersistence module $F$ and $v\in \bigcup_{a\in \R^2} F_a$, then $[v]^B\in k^{|B|}$ records the field coefficients in the unique expression of $v$ as a linear combination of $B$.  Let \[B.v = \{b_i \mid [v]^B_i\neq 0\}.\]
This definition of $B.v$ is in fact independent of the choice of order on $B$, so it extends to unordered bases $B$.  

\begin{lemma}
\label{Lem:Matrix_kernel}
Let $\gamma\colon F\to G$ be a morphism of finitely generated free bipersistence modules with ordered bases $B$ and $B'$, respectively, and let $T$ be the unlabeled matrix representing $\gamma$ with respect to $B$ and $B'$.  Then for any $v\in \bigcup_{a\in \R^2} F_a$, we have $v\in \ker \gamma$ if and only if $[v]^B\in \ker T$.
\end{lemma}
\begin{proof}
We have $\gamma(v)=0$ if and only if $[\gamma(v)]^{B'}=0$.  
But $[\gamma(v)]^{B'}=T[v]^B$, which gives the result.
\end{proof}

\begin{lemma}
\label{Lem:Kernel_grade}
For $\gamma\colon F\to G$ a morphism of finitely generated free persistence modules, $B$ a basis of $F$, $C$ a basis of $\ker \gamma$, and $c\in C$, we have \[\gr(c)=\bigvee_{b\in B.c}\gr(b).\]
\end{lemma}

\begin{proof}
Let $g=\bigvee_{b\in B.c}\gr(b)$, and observe that $\gr(c)\geq g$.  Consider the unique element $c'\in F_g$ such that $c=F_{g\subto\gr(c)}(c')$.  For any order on $B$, we have $[c']^B=[c]^B$, so $c'\in \ker \gamma$ by \cref{Lem:Matrix_kernel}.  Thus, since $c$ is an element of a basis of $\ker \gamma$ and a basis of a free module is a minimal generating set, we must have that $c=c'$.  This implies in particular that $\gr(c)=g$, as desired.  
\end{proof}

For $F$ a bipersistence module and $v\in \bigcup_{a\in \R^2} F_a$, let us write the $x$- and $y$-coordinates of $\gr(v)$ as $v_x$ and $v_y$, respectively.

\begin{lemma}\label{Lem:Injective_Maps}
Given a map $\gamma\colon F\to G$ of finitely generated free bipersistence modules, a basis $B$ of $F$, and a basis $C$ of $\ker \gamma$, there exist injective maps $j^x,j^y\colon C\to B$ such that  for all $c\in C$, 
\[c_x=j^x(c)_x\quad\textup{ and }\quad c_y=j^y(c)_y.\]
\end{lemma}

Our proof of \cref{Lem:Injective_Maps} will make light use of the language of Gr\"obner bases, which we now review in our special case.  Let $F$ be a finitely generated, free bipersistence module and fix an ordered basis $B$ of $F$.  For $v\in \bigcup_{a\in \R^2} F_a$, define the \emph{leading component} of $v$ to be $\max\,\{b\in B.v\}$.  For $F'\subset F$ a free submodule of $F$, a basis $C$ of $F'$ is called a \emph{Gr\"obner basis} (with respect to $B$) if no two distinct elements of $C$ have the same leading component in $F$.  It is easily checked that this definition agrees with the usual, more general definition of Gr\"obner basis, though we will not make use of that agreement.

We say $B$ is \emph{colexicographically ordered} if for all $b<b'\in B$, either $b_y<b'_y$ or $b_y=b'_y$ and $b_x\leq b'_x$.  

\begin{proposition}[{\cite[Remark 3.5]{lesnick2019computing}}]\label{Prop:Kernels_GBs}
If $\gamma\colon F\to G$ is a morphism of finitely generated, free bipersistence modules and $B$ is a colexicographically ordered basis of $F$, then there exists a basis $C$ of $\ker \gamma$ which is also a Gr\"obner basis with respect to $B$.
\end{proposition}

\begin{proof}
The corresponding result for $\Z^2$-persistence modules is proven algorithmically in \cite{lesnick2019computing}.  Given this, the statement for $\R^2$-persistence modules follows readily by expressing $\gamma$ as a left Kan extension along a grid function via \cref{Lem:Fin_Gen_LKE}, and then applying \cref{Lem:Exact_Extension} and \cref{Prop:Free_Modules_and_Extension}.
\end{proof}

\begin{proof}[Proof of \cref{Lem:Injective_Maps}]
By symmetry, it suffices to show that there exists a map $j^y$ as in the statement of the lemma.  Given any ordered basis $B$ of $F$ and any basis $C$ of $\ker \gamma$, we define a map $j^y\colon C\to B$ by 
\[j^y(c)=\, \max\, \{b\in B.c\mid b_y=c_y\}.\]

The definition of $j^y$ is well formed, because it follows from \cref{Lem:Kernel_grade} that the set over which we take the maximum is non-empty.  By definition, $c_y= j^y(c)_y$ 
for all $c\in C$.  

In general, $j^y$ is not necessarily injective.  But if $B$ is colexicographically ordered and $C$ is a basis of $\ker \gamma$ which is also a Gr\"obner basis with respect to $B$ (such $C$ exists by \cref{Prop:Kernels_GBs}), then $j^y$ is injective; indeed, the fact that $C$ is a Gr\"obner basis implies that the leading components in $F$ of the elements of $C$ are unique, and by the colexicographic choice of ordering on $B$, the leading component $b$ of each $c\in C$ satisfies $b_y=c_y$.   

This establishes that for any basis $B$ of $F$, a map $j^y:C\to B$ with the desired properties exists for one particular choice of basis $C$ of $\ker f$.  But then by \cref{prop:Basis_Grades}, the desired map exists for all choices of $C$, as claimed.  \end{proof}

\begin{example}\label{Ex:Maps_jx_and_jy}
Consider the map $\delta^f_1\colon \C_1(f)\to \C_0(f)$ of \cref{Ex:Stability}.  Let $B=(b_1,b_2,b_3)$ be an ordered basis of $\C_1(f)$ with \[\gr(b_1)=(1,4), \quad \gr(b_2)=(3,3),\quad\textup{ and }\quad\gr(b_3)=(4,1).\]  Let $C=(c_1,c_2)$ be an ordered basis of $\ker \partial_1=H_1\mathcal{S}(f)$ with $\gr(c_1)=(3,4)$ and $\gr(c_2)=(4,3)$.  \cref{Lem:Injective_Maps} guarantees the existence of maps $j^x,j^y\colon C\to B$ preserving $x$- and $y$-coordinates, respectively.  The only possibility is \[j^y(c_1)=b_1,\quad j^x(c_1)=j^y(c_2)=b_2,\quad \textup{ and } \quad j^x(c_2)=b_3.\]

In this example, one can check that if we make a small change in the $x$-coordinate of an element $b\in B$, then this changes $H_1\mathcal{S}(f)$ only though a corresponding change in the $x$-coordinate of $(j^x)^{-1}(b)$; the analogous statement for $j^y$ is also true.  For example, if we change $\gr(b_2)_x$ from $3$ to $2$, then $\gr(c_1)_x$ changes from 3 to 2.  And if we change $\gr(b_1)_x$ from $1$ to $0$, this causes no change to $H_1\mathcal{S}(f)$, because $(j^x)^{-1}(b_1)=\emptyset$.  
This illustrates the role that the injections $j^x$ and $j^y$ play in the proof of  \cref{Thm:Universality}\,(i): In general, under suitable assumptions, a perturbation to the $x$-grade of a basis element $b$ of $\C_i(f)$ causes only a corresponding perturbation to the $x$-grade of $(j^x)^{-1}(b)$.  Again, the same is true if we replace $x$ with $y$.  In what follows, we make this precise by using Kan extensions to encode the perturbations.
\end{example}

For $B=(b_1,\dots,b_{|B|})$ a finite ordered basis of a free bipersistence module, let $\gr(B)\colon\{1,\ldots,|B|\} \to \R^2$ denote the function sending $i$ to $\gr(b_i)$.  

\begin{lemma}\label{Lemma:Stability_of_Kernel2}
Suppose we are given a morphism of finitely generated free bipersistence modules $\gamma\colon F\to G$, a grid function $\G\colon \R^2\to \R^2$, and ordered bases $B$ and $C$ of  $F$ and $\ker \gamma$, respectively.  Let $B'=\Lan_{\G}(B)$.  Then $C$ induces an ordered basis $C'$ of $\ker \Lan_{\G}(f)$ with $|C'|=|C|$, such that
\begin{itemize}
\item [(i)] the inclusions $\ker \gamma \hookrightarrow F$ and $\ker \Lan_{\G}(\gamma) \hookrightarrow \Lan_{\G}(F)$ are represented by the same unlabeled matrix with respect to the bases $B$, $B'$, $C$, and $C'$,  
\item[(ii)] $\|\gr(C)-\gr(C')\|_p\leq \|\gr(B)-\gr(B')\|_p.$
\end{itemize}
\end{lemma}

\begin{proof}
By \cref{Lem:Exact_Extension}, $\ker \Lan_{\G}(\gamma)=\Lan_{\G}(\ker \gamma)$.  Thus, by \cref{Prop:Free_Modules_and_Extension}\,(ii), we may take $C'= \Lan_{\G}(C)$ (see \cref{Notation:Kan_Extended_Basis}).  Then (i) follows immediately from \cref{Prop:Free_Modules_and_Extension}\,(iii).

We now prove (ii).  Let $j^x,j^y\colon C\to B$ be injections as in the statement of \cref{Lem:Injective_Maps}.  For $b\in B$, let $b'$ denote the corresponding element of $B'$, and for $c\in C$, let $c'$ denote the corresponding element of $C'$. 
We claim that for all $c\in C$, 
\[
c'_x=j^x(c)'_x\quad\textup{ and }\quad c'_y=j^y(c)'_y.
\]
We will check the first equality; the proof of the second one is the same.  

For each $c\in C$, we have $[c']^{B'}=[c]^B$, as the vectors are equal to the same column of the matrix in (i). Thus \[B'.c'=\{b'\mid b\in B.c\}. \] Since $C$ and $C'$ are bases of $\ker \gamma$ and $\ker \Lan_\G(\gamma)$, it follows from \cref{Lem:Kernel_grade} that 
\begin{align*}
\gr(c)&=\bigvee_{b\in B.c}\gr(b),\\
\gr(c')&=\bigvee_{b'\in B'.c'}\gr(b')=\bigvee_{b\in B.c}\gr(b').
\end{align*}
Therefore,  \[c_x=\max_{b\in B.c} b_x\quad\textup{ and }\quad c'_x=\max_{b\in B.c} b'_x.\]  By the definition of $j^x$, we have 
$j^x(c)_x=c_x$, so \[j^x(c)_x=\max_{b\in B.c} b_x.\]  Since $\G_1$ is order-preserving  and $b'_x=\G_1(b_x)$ for all $b\in B$, the bijection $B\to B'$ mapping $b$ to $b'$ preserves the order of $x$-coordinates.  
Therefore, \[j^x(c)'_x=\max_{b\in B.c} b'_x=c'_x,\]
as claimed.

It now follows that for all $c\in C$, we have 
\begin{align*}
| c_x-c'_x |=|j^x(c)_x-j^x(c)'_x|,\\
| c_y-c'_y |=|j^y(c)_y-j^y(c)'_y|.
\end{align*}
This implies that 
\begin{align*}
\|\gr(C)-\gr(C')\|_p&=\left(\sum_{c\in C} | c_x-c'_x |^p+| c_y-c'_y |^p \right)^{\frac{1}{p}}\\
&=       \left(\sum_{c\in C} |j^x(c)_x-j^x(c)'_x|^p+|j^y(c)_y-j^y(c)'_y|^p \right)^{\frac{1}{p}}\\
&\leq \left(\sum_{b\in B} | b_x-b'_x |^p+| b_y-b'_y |^p \right)^{\frac{1}{p}}\\
&= \|\gr(B)-\gr(B')\|_p,
\end{align*}
where the inequality follows from the injectivity of $j^x$ and $j^y$.
\end{proof}

Recall from \cref{Def:Order_Refines} that for functions $f,g\colon S\to \R$, we say $g$ is $f$-compatible if $f(x)\leq f(y)$ implies $g(x)\leq g(y)$ for all $x,y\in S$.  We extend this definition to functions $f,g:S\to \R^2$ by declaring that $g$ is $f$-compatible if for each $i\in \{1,2\}$, $g^i$ is $f^i$-compatible, where $f^i,g^i\colon S\to \R$ are the $i^{\mathrm{th}}$ component functions of $f$ and $g$.

\begin{proof}[Proof of \cref{Thm:Universality}\,(i)]
We first prove the result for the case $n=2$, and then briefly discuss the easy adaptation of the proof to the case $n=1$.
Let $X$ be a finite CW-complex and $f,g\colon\Cells(X)\to \R^2$ be a pair of monotone functions.  For $t\in [0,1]$, define $h_t\colon \Cells(X)\to \R^2$ by linear interpolation between $f$ and $g$, i.e.,\[h_t(\sigma)=(1-t)f(\sigma)+tg(\sigma)\]
for all $\sigma\in \Cells(X)$.  As in the interpolation argument in the proof of  \cref{Thm:Main_Thm_Intro}\,(iv), there exists a finite set of real numbers \[0=t_0<t_1<t_2<\dots<t_{w+1}=1\]
such that for $i\in \{0,\ldots, w\}$ and any $s_i\in (t_i,t_{i+1})$, both $h_{t_i}$ and $h_{t_{i+1}}$ are $h_{s_i}$-compatible; we may take $t_1,\ldots,t_w$ to be the set of points $t\in (0,1)$ where the partial order of the values of $h_{t}$ changes as $t$ increases.  
Now \[\|f-g\|_p=\sum_{i=0}^w \|h_{t_i}-h_{s_i}\|_p+ \|h_{s_i}-h_{t_{i+1}}\|_p.\]
Thus, by the triangle inequality for $d_{\mcI}^p$, it suffices to prove the result in the special case that $g$ is $f$-compatible.  

Assuming that $g$ is $f$-compatible, it is easy to check that there exists a grid function $\G\colon \R^2\to \R^2$ such that $\mathcal S(g)= \Lan_{\G}(\mathcal S(f))$;  indeed, we may take $\G$ to be any grid function such that $\G(f(\sigma))=g(\sigma)$ for all $\sigma\in \Cells(X)$.  For such $\G$, we then also have 
\[\C_j(g)=\Lan_{\G}(\C_j(f))\quad \textup{and}\quad \partial_j^g=\Lan_{\G}(\partial_j^f)\]
for all $j$.   

By the way $\C(f)$ is defined, each $\C_j(f)$ comes equipped with a distinguished basis $B_j^{f}$.  We choose an arbitrary order on each $B_j^{f}$.  Let $B_j^g=\Lan_{\G}(B_j^f)$ be the corresponding basis of $\C_j(g)$, and let $Z^f$ be an ordered basis of $\ker \partial_j^f$.  \cref{Lemma:Stability_of_Kernel2}\,(i) provides an ordered basis $Z^g$ of $\ker \partial_j^g$ which is compatible with $Z^f$, in the sense that the inclusions $\ker \partial_j^f\hookrightarrow \C_j(f)$ and $\ker \partial_j^g\hookrightarrow \C_j(g)$ are represented with respect to the these bases by the same unlabeled matrix. 

Since $\C(f)$ and $\C(g)$ are chain complexes, we know that $\im \partial^f_{j+1}\subset \ker \partial_j^f$ and $\im \partial^g_{j+1}\subset \ker \partial_j^g$. Let $\tilde \partial^f_{j+1}$ and $\tilde \partial^g_{j+1}$ denote the respective restrictions of $\partial^f_{j+1}$ and $\partial^g_{j+1}$ to the codomains $\ker \partial^f_{j}$ and $\ker \partial^g_{j}$.  Let $P^f$ denote the the matrix representation of $\tilde \partial^f_{j+1}$  with respect to the bases $Z^f$ and $B_{j+1}^f$, and similarly, let $P^g$ denote the the matrix representation of $\tilde \partial^g_{j+1}$  with respect to $Z^g$ and $B_{j+1}^g$.  Clearly, $P^f$ and $P^g$ are presentations of $H_j\mathcal S(f)$ and $H_j\mathcal S(g)$, respectively, and it is easily checked that the two presentations have the same underlying matrix.

By \cref{Lemma:Stability_of_Kernel2}\,(ii), the $\ell^p$-distance on the row labels of $P^f$ and $P^g$ is at most the $\ell^p$-distance between the restrictions of $f$ and $g$ to the $j$-cells of $X$.  And the $\ell^p$-distance on the column labels of $P^f$ and $P^g$ is exactly the $\ell^p$-distance between the restrictions of $f$ and $g$ to the $(j+1)$-cells of $X$, because the column labels and function values agree.  Thus, \[d^p(P^f,P^g)\leq \|f^j-g^j\|_p,\] where $f^j$ denotes the restriction of $f$ to the $j$- and $(j+1)$-cells of $X$, and $g^j$ is defined analogously.  
Since \[d_{\mcI}^p(H_j\mathcal S(f),H_j\mathcal S(g))\leq d^p(P^f,P^g),\] we have that 
\begin{equation}\label{Eq:p-Stability}
d_{\mcI}^p(H_j\mathcal S (f),H_j\mathcal S(g)) \leq \|f^j-g^j\|_p\leq \|f-g\|_p,
\end{equation}
 which shows that $d_{\mcI}^p$ is $p$-stable.  

To see that $d_{\mcI}^p$ is also $p$-stable across degrees with constant $2^{\frac{1}{p}}$, let $\|f^*-g^*\|_p$ denote the sequence of real numbers 
\[\|f^0-g^0\|_p,\ \|f^1-g^1\|_p,\ \|f^2-g^2\|_p,\ \ldots\]
and simply note that 
\[d_{\mcI}^p(H_*\mathcal S (f),H_*\mathcal S(g))\leq \|f^*-g^*\|_p\leq 2^{\frac{1}{p}}\|f-g\|_p,\]
where the first inequality follows from the first inequality of \eqref{Eq:p-Stability}.

Finally, we consider the case $n=1$.  A simplification of the proof we have given shows that in this case, $d_{\mcI}^p$ is $p$-stable, and also $p$-stable across degrees with constant $2^{\frac{1}{p}}$.  Alternatively, these results can be shown via a reduction to the $n=2$ case.  To obtain the stronger result that $d_{\mcI}^p$ is $p$-stable across degrees with constant $1^{\frac{1}{p}}=1$, we need an additional idea, which we now outline.

Let $f\colon\Cells(X)\to \R$ be a monotone function.  Carlsson and Zomorodian \cite{zomorodian2005computing} have observed that a basis $B_j$ of each chain module $\C_j(f)$ can be chosen such that with respect to these bases, each matrix $[\partial^f_j]$ is in graded Smith normal form, i.e., $[\partial^f_j]$ has at most one non-zero entry in each row and each column.  The elements of $B_j$ corresponding to zero columns of $[\partial^f_j]$ then form a basis $B_j^{\ker}$ of $\ker \partial^f_j$, and the submatrix of $[\partial^f_{j+1}]$ consisting of columns  indexed by $B_{j+1}\setminus B_{j+1}^{\ker}$ and rows indexed by $B_{j}^{\ker}$ determines a presentation matrix $P_j$ for $H_{j}\mathcal S (f)$.  Note that for each $j$, the rows of $P_{j}$ are indexed by $B_j^{\ker}$ and the columns of $P_{j-1}$ are indexed by $B_{j}\setminus B_{j}^{\ker}$.  Thus, if one adapts our proof from the $n=2$ case to the $n=1$ case, using such presentations $P_j$ throughout, then it is easy to see that a change to the function value of a $j$-simplex in $\Cells(X)$ can induce a corresponding change to the label of either a row of $P_{j}$ or a column of $P_{j-1}$ but (in contrast to the 2-parameter case,) not to both simultaneously.  
Using this observation, our argument above for the 2-parameter case then adapts easily to give the claimed bound.
\end{proof}

\subsection{Tightness of \cref{Thm:Universality}\,(i)}
To establish our tightness result, we will need a generalization of the definition of $\delta$-interleavings from \cref{Sec:Interleavings}: For $v\in [0,\infty)^n$, define the \emph{$v$-interleaving category} $\I^v$ to be the thin category with object set $\R^n\times \{0,1\}$ and a morphism $(a,i)\to (b,j)$ if and only if either  
\begin{enumerate}
\item $a+v \leq b$, or
\item $i=j$ and $a\leq b$.
\end{enumerate}

\begin{definition}
A \emph{$v$-interleaving} between $\R^n$-persistence modules $M$ and $N$ is a functor
\[
Z\colon \I^v\to \Vect
\]
such that $Z\circ {\mathcal J}^0=M$ and $Z\circ {\mathcal J}^1=N$, where  ${\mathcal J}^0$ and ${\mathcal J}^1$ are as defined in \cref{Sec:Interleavings}.
\end{definition}

\begin{definition}
Let $d_\mcI^+$ be the distance on $\R^n$-persistence modules defined by 
\[d_\mcI^+(M,N) = \inf\,\left\{v_1+\cdots+ v_n \mid  \textup{$\exists$ a $v$-interleaving between $M$ and $N$}\right\}.\]
\end{definition}
This is indeed a distance: If $M$ and $M'$ are $(v_1,\dots,v_n)$-interleaved, and $M'$ and $M''$ are $(v'_1,\dots,v'_n)$-interleaved, then $M$ and $M''$ are $(v_1+v'_1,\dots,v_n+v'_n)$-interleaved, which implies the triangle inequality.

\begin{lemma}\label{Lem:Tightness}
If $M$ and $N$ are finitely presented $\R^n$-persistence modules, then
\[d_\mcI^+(M,N) \leq n^{1-\frac{1}{p}} d_\mcI^p(M,N)\]
for all $p\in [1,\infty]$.
\end{lemma}

\begin{proof}
We will prove the inequality with $d_\mcI^p$ replaced by $\hat d_\mcI^p$.  Given this, $n^{\frac{1}{p}-1}d_\mcI^+$ is a distance bounded above by $\hat d_\mcI^p$, so the inequality in the lemma follows from \cref{Prop:Largest_Lower_Bound}\,(ii).

Suppose $P$ and $P'$ are presentations of $M$ and $N$ with the same underlying $r\times c$ matrix. Recall that by definition, \[d^p(P,P') = \|\labels(P)- \labels(P')\|_p,\] where we regard $\labels(P) - \labels(P')$ as a function from $\{1,2,\ldots,r+c\}$ to $\R^n$.  Let $L_i$ denote the $i^{\mathrm{th}}$ component of this function, and 
define $v \in \R^n$  by $v_i =\max |L_i|$.  
Observe that for any $j\in \{1,2,\ldots,r+c\}$, we have 
\[-v\leq \labels(P)_j - \labels(P')_j\leq v.\] 
Using this, it is easy to check that $M$ and $N$ are $v$-interleaved, which gives \[d_\mcI^+(M,N)\leq \sum_{i=1}^n v_i.\] Moreover, we have
\begin{align*}
\sum_{i=1}^n v_i &= \|v\|_1\\
&\leq n^{1-\frac{1}{p}}\|v\|_p\\
&\leq n^{1-\frac{1}{p}}\|\labels(P) - \labels(P')\|_p\\
&= n^{1-\frac{1}{p}}d^p(P,P'),
\end{align*}
where the first inequality follows from the standard relation between $p$-norms on $\R^n$. Since $P$ and $P'$ were arbitrary, this finishes the proof.
\end{proof}

\begin{proposition}\label{Prop:Tightness}
For all $p\in [1,\infty]$ and $n\in \{1,2\}$, $d_\mcI^p$ is not $p$-stable across degrees with constant $c$ on $\R^n$-persistence modules for any $c<n^\frac{1}{p}$.
\end{proposition}
\begin{proof}
Note that $d_{\mcI}^+=d_{\mcI}$ on $\R$-persistence modules.  Thus, \cref{Lem:Tightness} implies that for any functions $f,g:X\to \R$ on a CW-complex $X$ consisting only of a single $0$-cell, 
\[d_\mcI^p(H_0\mathcal S(f),H_0\mathcal S(g))\geq d_\mcI(H_0\mathcal S(f),H_0\mathcal S(g))= \|f-g\|_p.\]
Taking $f\ne g$, the case $n=1$ of the result follows.

For the $n=2$ case, consider the filtrations $f$ and $g$ of \cref{Ex:Stability}.  Note that $\|f-g\|_p = \|(1,1)\|_p$. We claim that 
$d_\mcI^+(H_i\mathcal S(f),H_i\mathcal S(g))\geq 2$ for both $i=0$ and $i=1$. For $i=1$, this follows from the observation that $H_1\mathcal S(f)$ and $H_1\mathcal S(g)$ are not $v$-interleaved for any $v\ngeq(1,1)$. For $i=0$, observe that if $v=(v_1,v_2)$ and $v_1+v_2<2$, then $H_0\mathcal S(f)_{(2,2)-v\subto (2,2)+v}$ has rank two. Thus, if $H_0\mathcal S(f)$ and $H_0\mathcal S(g)$ are $v$-interleaved, then $H_0\mathcal S(g)_{(2,2)}$ must have dimension at least two, which is false.  
Applying \cref{Lem:Tightness}, we get $d_\mcI^p(H_i\mathcal S(f),H_i\mathcal S(g))\geq 2^\frac{1}{p}$. Thus,
\[\|d_\mcI^p(H_*\mathcal{S}(f),H_*\mathcal{S}(g))\|_p \geq \|(2^\frac{1}{p},2^\frac{1}{p})\|_p = 2^\frac{1}{p}\|f-g\|_p.\qedhere\]
\end{proof}

\subsection{Proof of Universality}
Our proof of \cref{Thm:Universality}\,(ii) is similar to the proof of \cref{Thm:Stability_and_Universality_of_The_Interleaving_Distance}\,(ii)  given in \cite{lesnick2015theory}, though simpler.  As with that result (and most similar universality results in the TDA literature, e.g., \cite{d2010natural,bauer2020universality,bauer2020reeb,blumberg2017universality}) the argument depends on a lifting result, which we now state:

\begin{lemma}\label{Lem:Universality_Lifting}
Let $M$ and $N$ be finitely presented bipersistence modules with coefficients in a prime field $k$.  Given presentation matrices $P_M$ and $P_N$ for $M$ and $N$ with the same unlabeled matrix, there exist a CW-complex $X$ and functions $f^M,f^N\colon \Cells(X)\to \R^2$ such that 
\begin{itemize}
\item $M\cong H_1\mathcal S(f^M)$ and $N\cong H_1\mathcal S(f^N)$,
\item $d^p(P_M,P_N)=\|f^M-f^N\|_p$.
\end{itemize}
\end{lemma}

\begin{proof}
Let $T$ denote the unlabeled matrix underlying $P_M$ and $P_N$, and assume that $T$ has $r$ rows and $c$ columns.  Let $X$ be a CW-complex with a single vertex $v$, 1-cells $\sigma_1,\ldots,\sigma_r$, and 2-cells $\tau_1,\ldots, \tau_c$, with the $2$-cells attached so as to satisfy the following conditions: If $k$ is finite, then the degree-2 cellular boundary matrix $\partial_2$ is equal to $T$; and if $k=\mathbb Q$, then $\partial_2$ is obtained from $T$ by multiplying each column of $T$ by an integer.  

We define $f^M,f^N\colon \Cells(X)\to \R^2$ by letting \[f^M(v)=f^N(v)=\min(\im \labels(P_M) \cup \im \labels(P_N)),\] and (using the notation for row and column labels introduced above \cref{Prop:SNF}) letting
\[
\begin{aligned}[t]
f^M(\sigma_i)&=\labels(P_M)_{i*},\\
f^M(\tau_j)&=\labels(P_M)_{*j},
\end{aligned}
\quad
\begin{aligned}[t]
f^N(\sigma_i)&=\labels(P_N)_{i*}\\
f^N(\tau_j)&=\labels(P_N)_{*j}
\end{aligned}
\quad 
\begin{aligned}[t]
 \textup{ for $i\in \{1,\dots,r\}$},\\
  \textup{ for $j\in \{1,\dots,c\}$}.
  \end{aligned}
\]
It is easy to see that $M\cong H_1(f^M)$, $N\cong H_1(f^N)$, and $d^p(P_M,P_N)=\|f^M-f^N\|_p$.
\end{proof}   
   
\begin{proof}[Proof of \cref{Thm:Universality}\,(ii)]
We give the proof in the case $n=2$; the proof when $n=1$ is essentially the same.

Let $d$ be a $p$-stable distance on finitely presented bipersistence modules. We need to show that $d\leq d_{\mcI}^p$.  By \cref{Prop:Largest_Lower_Bound}\,(ii), it in fact suffices to show that $d\leq \hat d_{\mcI}^p$. Consider two finitely presented bipersistence modules $M$ and $N$ such that $\hat d_{\mcI}^p(M,N)<\infty$.  For all $\delta> \hat d_{\mcI}^p(M,N)$, there exist presentations $P_M$ and $P_N$ for $M$ and $N$ with $d^p(P_M, P_N)\leq \delta$.  By \cref{Lem:Universality_Lifting} and the $p$-stability of $d$, we have that $d(M,N)\leq \delta$.  Since this holds for all $\delta>\hat d_{\mcI}^p(M,N)$, we see that $d(M,N)\leq \hat d_{\mcI}^p(M,N)$.  
\end{proof}

\begin{remark}[Universality for Non-Prime Fields]\label{Rem:Non_Prime_Universality}
As noted in the introduction, the condition that the underlying field $k$ be prime is in fact unnecessary for the 1-parameter version of \cref{Thm:Universality}\,(ii).  In the 2-parameter case, the question of whether \cref{Thm:Universality}\,(ii) holds for arbitrary fields $k$ is an interesting open question.  In fact, the analogous question about the interleaving distance on multiparameter persistence modules is also open \cite{lesnick2015theory}.
\end{remark}

\begin{remark}[Stability and Universality for Three or More Parameters]\label{Rem:Three_Parameters_or_More}
Our proof of the stability result \cref{Thm:Universality}\,(i) makes essential use of \cref{Lem:Free_Kernels}, which says that kernels of morphisms between finitely generated free bipersistence modules are free.  \cref{Lem:Free_Kernels} does not hold for 3-parameter persistence modules, and as such, we do not expect the 1-Lipchitz stability bound \cref{Thm:Universality}\,(i) to extend to $n$-parameter persistence modules, for $n\geq 3$.  However, we conjecture that for any fixed $n\geq 3$, a weaker Lipschitz stability bound does hold for $n$-parameter persistence modules.  It is not clear whether one would have a universality result in the general setting.
\end{remark}

\section{Conclusion}
In this work, we have introduced two multiparameter generalizations of the $p$-Wasserstein distance on barcodes, the $p$-presentation distance and the $p$-matching distance, which for  $p=\infty$ are equal to the interleaving distance and the matching distance, respectively.  We have shown that several fundamental properties of the Wasserstein, interleaving, and matching distances extend to our new distances.   As explained in \cref{Sec:Multicover_Outline}, in a forthcoming companion paper \cite{bjerkevikMulticover} we apply our distances to refine the stability theory for multicover persistent homology developed in \cite{blumberg2020stability}.  

Our work raises a number of mathematical questions that would be worth exploring in future work.  We have already mentioned several of these in \cref{Conj:Computational_Conjectures}, \cref{Rem:Non_Prime_Universality}, and \cref{Rem:Three_Parameters_or_More}.  We conclude by discussing three others.

First, one can define the interleaving distance not only on persistence modules, but also on (multi)filtered topological spaces.  Modifying the definition slightly, one can also define a homotopy invariant version \cite{blumberg2017universality}.  Both versions satisfy universal properties analogous to the one for modules \cite{blumberg2017universality,lesnick2012multidimensional}.  Given our results, it is natural to wonder whether we can define $\ell^p$-type generalizations of these distances  which satisfy similar universality properties. 

Second, a natural direction would be to extend our 2-parameter stability result to the \emph{multicritical} setting, where instead of considering sublevel filtrations of monotone functions on a CW-complex, we consider bifiltrations on a CW-complex in which a cell can be born at multiple incomparable grades.  While many interesting cellular bifiltrations do arise as sublevel filtrations, e.g., the function-Rips bifiltrations \cite{carlsson2009theory}, rhomboid bifiltrations  \cite{corbet2021computing}, and subdivision bifiltrations \cite{sheehy2012multicover,blumberg2020stability}, others of interest in applications, such as the degree bifiltrations \cite{lesnick2015interactive,blumberg2020stability}, are multicritical.  

Third, our main results are stated only for finitely presented modules.  The problem of extending them to a larger class of modules, e.g., to \pfd modules, is interesting.  To obtain such extensions, one would first need to extend the definition of the $p$-presentation distance, and the question of how to best do this seems to be a subtle one.  

\bibliographystyle{abbrv}
\bibliography{WI_Refs}
\end{document}